\documentclass[a4paper,final]{amsart}
\pdfoutput=1

\usepackage[utf8]{inputenc}
\usepackage[T1]{fontenc}
\usepackage{amsmath}
\usepackage{amsthm}
\usepackage{amssymb}
\usepackage{amsfonts}
\usepackage{booktabs,dcolumn}
\usepackage{enumitem,graphicx}
\usepackage{mathtools}
\usepackage[final]{hyperref}
\usepackage{subfig}
\usepackage[notref,notcite]{showkeys}

\usepackage[backend=biber,maxbibnames=5,maxalphanames=5,style=numeric,bibencoding=utf8,giveninits,url=false,isbn=false]{biblatex}
\AtBeginBibliography{\small}

\usepackage{stmaryrd}	

\numberwithin{equation}{section}

\usepackage{xcolor}
\usepackage{todonotes}
\usepackage{changes}
\definechangesauthor[name=Christian, color=blue]{CK}
\definechangesauthor[name=Adrian, color=green]{AH}
\definechangesauthor[name=Lars, color=red]{LD}

\IfFileExists{dsfont.sty}%
{\usepackage{dsfont}\newcommand{\R}{\mathds{R}}\newcommand{\N}{\mathds{N}}\newcommand{\PS}{\mathds{P}}}%
{\newcommand{\R}{\mathbb{R}}\newcommand{\N}{\mathbb{N}}\newcommand{\PS}{\mathbb{PS}}}

\newcommand{\D}{\Omega}

\newcommand{\Q}{\mathrm{Q}}
\newcommand{\Z}{\mathbf{Z}}

\newcommand{\CRd}{\operatorname{\mathbf{CR}}}
\newcommand{\CR}{\operatorname{CR}}

\newcommand{\T}{\mathcal{T}}
\newcommand{\Neigh}{\mathcal{N}}

\newcommand{\FC}{\mathcal{F}}
\newcommand{\FCi}{\mathcal{F}_i}
\newcommand{\FCb}{\mathcal{F}_{\partial}}

\newcommand{\ep}{\epsilon}

\newcommand{\A}{\boldsymbol{A}}
\newcommand{\F}{\boldsymbol{F}}

\newcommand{\I}{\mathcal{I}}
\newcommand{\ICR}{\ensuremath{\I^{\textsf{cr}}}}
\newcommand{\J}{\mathcal{J}}

\newcommand{\CC}{\boldsymbol{C}}
\newcommand{\EE}{\boldsymbol{E}}
\newcommand{\LL}{\boldsymbol{L}}
\newcommand{\QQ}{\boldsymbol{Q}}
\newcommand{\PP}{\boldsymbol{P}}
\newcommand{\WW}{\boldsymbol{W}}
\newcommand{\BB}{\boldsymbol{B}}

\newcommand{\nn}{\boldsymbol{n}}
\newcommand{\uu}{\boldsymbol{u}}
\newcommand{\vv}{\boldsymbol{v}}
\newcommand{\ww}{\boldsymbol{w}}
\newcommand{\ff}{\boldsymbol{f}}
\newcommand{\qq}{\boldsymbol{q}}
\newcommand{\pp}{\boldsymbol{p}}
\newcommand{\xx}{\boldsymbol{x}}
\newcommand{\yy}{\boldsymbol{y}}

\renewcommand{\d}{\mathrm{d}}
\newcommand{\DB}{\boldsymbol{\mathcal{D}}}
\newcommand{\DD}{\boldsymbol{\mathrm{D}}}
\newcommand{\nablah}{\nabla_{\!h}}

\newcommand{\divergenz}{\operatorname{div}}

\newcommand{\divh}{{\divergenz}_h}

\newcommand{\wh}{\widehat}

\newcommand{\w}{\omega}

\newcommand{\zero}{\boldsymbol{0}}
\newcommand{\bzero}{{\mathbf{0}}}

\renewcommand{\phi}{\varphi}
\renewcommand{\epsilon}{\varepsilon}

\newcommand{\norm}[1]{{\left\lVert{#1}\right\rVert}}

\newcommand{\abstmp}[2]{{#1\lvert{#2}#1\rvert}}
\newcommand{\abs}[1]{\abstmp{}{#1}}
\newcommand{\bigabs}[1]{\abstmp{\big}{#1}}

\newcommand{\meantmp}[2]{#1\langle{#2}#1\rangle}
\newcommand{\mean}[1]{\meantmp{}{#1}}

\providecommand{\settmp}[2]{{#1\{{#2}#1\}}}
\providecommand{\set}[1]{\settmp{}{#1}}

\newcommand{\Avg}[1]{\left \{\!\!\left \{ #1 \right \}  \!\!\right \}}
\newcommand{\Jump}[1]{\left \llbracket #1 \right \rrbracket}

\def\Xint#1{\mathchoice
   {\XXint\displaystyle\textstyle{#1}}%
   {\XXint\textstyle\scriptstyle{#1}}%
   {\XXint\scriptstyle\scriptscriptstyle{#1}}%
   {\XXint\scriptscriptstyle\scriptscriptstyle{#1}}%
   \!\int}
\def\XXint#1#2#3{{\setbox0=\hbox{$#1{#2#3}{\int}$}
     \vcenter{\hbox{$#2#3$}}\kern-.5\wd0}}

\def\dashint{\Xint-}

\theoremstyle{plain}
\newtheorem{theorem}{Theorem}[section]
\newtheorem{proposition}[theorem]{Proposition}
\newtheorem{lemma}[theorem]{Lemma}

\newtheorem{assumption}[theorem]{Assumption}
\newtheorem{remark}[theorem]{Remark}
\newtheorem{corollary}[theorem]{Corollary}
\newtheorem{method}[theorem]{Method}

\newcounter{descr}
  \def\descrhead{\refstepcounter{descr}\begin{description}\item[{\textbf{\thedescr}}] }
  \def\descrfin{ \end{description}}

\providecommand{\Ruzicka}{{R{\r u}{\v z}i{\v c}ka}}

\begin{document}

\title[Pressure robust discretizations of the nonlinear Stokes equations]
{Pressure robust finite element discretizations of the nonlinear Stokes equations}

\begin{abstract}
  We present first-order nonconforming Crouzeix-Raviart discretizations for the
  nonlinear generalized Stokes equations with
  $(r,\epsilon)$-structure. Thereby the velocity-errors are
  independent of the pressure-error; i.e., the method is pressure
  robust. This improves suboptimal rates previously
    experienced for non pressure robust methods.
\end{abstract}

\author[L. Diening]{Lars Diening}
\address[Lars Diening]{Universität Bielefeld \\
Fakultät für Mathematik \\ D-33615 Bielefeld \\ Germany}
\email{lars.diening@uni-bielefeld.de}

\author[A. Hirn]{Adrian Hirn}
\address[Adrian Hirn]{Hochschule Esslingen, Robert-Bosch-Stra\ss{}e 1, 73037 G\"oppingen, Germany}
\email{adrian.hirn@hs-esslingen.de}

\author[C. Kreuzer]{Christian Kreuzer}
\address[Christian Kreuzer]{TU Dortmund \\ Fakult{\"a}t f{\"u}r Mathematik \\ D-44221 Dortmund \\ Germany}
\email{christian.kreuzer@tu-dortmund.de}

\author[P. Zanotti]{Pietro Zanotti}
\address[Pietro Zanotti]{Universit\`{a} degli Studi di Milano\\
  Dipartimento di Matematica \\20133 Milano\\ Italy}
\email{pietro.zanotti@unimi.it}



\date{\today}
\maketitle

\section{Introduction}
\label{sec:introduction}

In this article, we study the finite element discretization of the nonlinear Stokes equations in a bounded domain $\D\subset\R^d$, $d\in\{2,3\}$,
\begin{align}\label{intro:eq1}
  -\divergenz \A(\DB \uu)+\nabla p = \ff
  \quad\text{and}\quad \divergenz \uu =0  \quad \text{in }\D,
  \qquad \uu=0 &&\text{on }\partial\D.
\end{align}
Here $\uu$ denotes the velocity, $p$ the kinematic pressure, and $\ff$
an external body force.  The symbol $\DB \uu$ stands either for the
velocity gradient $\nabla\uu$ or for its symmetric part
$\DD\uu=\frac12(\nabla\uu+(\nabla\uu)^T)$.  For given $r>1$, $\ep\geq
0$ the nonlinear function $\A:\R^{d\times d}\to\R^{d\times d}$
is supposed to have $(r,\ep)$-structure (or more
  generally $\phi$-structure; see Assumption \ref{assumption:nonlinear_operator}).  A prototypical example is
\begin{align}\label{intro:eq2}
  \A(\DB \uu)= \left(\ep^2+|\DB \uu|^2\right)^{\frac{r-2}{2}}\DB \uu\,.
\end{align}
For \(\epsilon=0\) the operator corresponds to the $r$-Laplacian, while for $\ep>0$ the operator regularizes the degeneracy of the $r$-Laplacian as the modulus of \(\DB\uu\) tends to zero. 
Equations with $(r,\ep)$–structure arise in various physical applications,
such as in the theory of plasticity, bimaterial problems in
elastic-plastic mechanics, non-Newtonian fluid mechanics, blood
rheology and glaciology; see
e.g. \cite{Liu_1999,Malek_1996,Helanow_2018}. 

The finite element (FE) approximation of equations with
$(r,\ep)$-structure has been widely studied, see e.g.
\cite{BarrettLiu:1993,Ruzicka:2006,Liu:2003,Barrett:1993,Barrett:1994,Belenki:2012,Hirn:2013}.
According to \cite{Ruzicka:2006,Belenki:2012}, it is
relevant to bound the error between $\uu$ and the
discrete velocity $\uu_h$ in the so-called natural- or \(\F\)-distance
$\norm{\F(\DB\uu)-\F(\DB\vv)}_2$ where $\F$ is a nonlinear vector
field adapted to the problem's $(r,\ep)$-structure; see~\eqref{def_F} for
the precise definition.  This error notion is equivalent to the
quasi-norm introduced by Barrett/Liu
\cite{BarrettLiu:1993,Barrett:1994}. 

For conforming exactly divergence-free discretizations 
C\'{e}a-type estimates
\begin{align}\label{intro:eq3}
  \norm{\F(\DB\uu)-\F(\DB\uu_h)}_2
  \lesssim  \inf_{\divergenz \vv_h=0}\norm{\F(\DB\uu)-\F(\DB\vv_h)}_2
\end{align}
as for the elliptic \(r\)-Laplace problem \cite[Theorem
2.1]{BarrettLiu:1993} can be obtained by straight forward application
of the elliptic techniques; see e.g.~\cite[Lemma 5.2]{Ruzicka:2006}.
We call a method satisfying~\eqref{intro:eq3} \textit{quasi-optimal
  and pressure-robust}: the velocity-error measured in the
natural distance is proportional to the corresponding best error 
and independent of the pressure.  Such discretizations have recently
received much attention for linear Stokes-equations; compare with
\cite{Guzman:2014,Guzman:2018,Scott:1985,Zhang:2023}.  In order to
replace the global approximation problem on the right hand side
of~\eqref{intro:eq3} by local approximation problems, usually local
Fortin operators are employed; compare
with~\cite[\S54.1]{Ern:2021}.  Such operators are typically not
perfectly local in the sense that the interpolated function on a mesh
element \(K\) depends on the function values in the neighbourhood
\(\omega_K\) of \(K\). This results in 
\begin{align}\label{intro:eq5}
  \norm{\F(\DB\uu)-\F(\DB\uu_h)}_2^2
  \lesssim  \sum_{K}\inf_{\QQ_K\in\R^{d\times d}}\norm{\F(\DB\uu)-\QQ_K}_{2;\omega_K}^2.
\end{align}
Corresponding results for the nonlinear Laplacian can be found in \cite{Ruzicka:2006}. Note that in the case of the linear Stokes problem, the neighbourhood \(\omega_K\) can be avoided using techniques from~\cite{Veeser:16} but for \(r\neq2\), properties of the natural distance prevent from a direct generalisation; compare with Remark~\ref{R:qo-discuss}. It is also noticeable that all above a priori bounds are limited to first order at most. Here again, the reason is the nonlinear structure of the natural distance, in particular, the use of Jensen's inequality while generalising stability properties to the natural distance.

For not exactly divergence-free discretizations, it was proved in~\cite{Belenki:2012} that
\begin{equation}\label{intro:eq6}
  \begin{aligned}
    \norm{\F(\DB\uu)-\F(\DB\uu_h)}_{2}^2
    &\lesssim \sum_{K}\inf_{\QQ_K\in\R^{d\times d}}\norm{\F(\DB\uu)-\QQ_K}_{2;\omega_K}^2 \\
    &+ \inf_{q_h\in\Q_h}\int_{\D}\left( |p-q_h|^2+|\DB\uu|^2\right)^{\frac{r'-2}{2}}|p-q_h|^2
  \end{aligned} 
\end{equation}
with $r'$ the H\"{o}lder conjugate of $r$. This estimate substantially differs from \eqref{intro:eq5} in that it is not pressure robust, 
i.e., the velocity-error is interfered by the pressure; compare
with~\cite{JohnLinkeMerdonNeilanRebholz:2018}.
In particular, for \(r\neq2\), this interference is of nonlinear nature and
may lead to suboptimal convergence rates even for regular solutions
with $\F(\DB\uu)\in W^{1,2}(\D)^{d\times d}$ and $p\in W^{1,r'}(\D)$. Indeed,
Belenki, Berselli, Diening and \Ruzicka{} \cite{Belenki:2012} derived
from~\eqref{intro:eq6} for 
the MINI element the the following error estimate in terms of the mesh-size \(h\)
\begin{align}\label{intro:eq4}
  \norm{\F(\DB\uu)-\F(\DB\uu_h)}_2
  \lesssim h^{\min\{1,\frac{r'}{2}\}}
  \,.
\end{align}
Similar results are derived by \cite{Hirn:2013} for a stabilized $\mathbb{Q}_1/\mathbb{Q}_1$-discretization. 
Comparing \eqref{intro:eq4} and~\eqref{intro:eq5}, 
we conclude that for $r>2$ the velocity-error estimate \eqref{intro:eq4} is
suboptimal with respect to the 
approximation properties of the discrete space. This is also confirmed by
numerical studies in \cite{Belenki:2012}.

A similar reasoning (i.e. the pressure-error depends on the
velocity-error)  limits the theoretical convergence of the pressure error in
\cite{Belenki:2012} to
\begin{align}\label{eq:pressure-error-old}
  \norm{p-p_h}_{r'}\lesssim h^{\min\{\frac{2}{r'},\frac{r'}{2}\}}.
\end{align}
In this case, however, the numerical results in~\cite{Belenki:2012}
show \(\mathcal{O}( h^{\min\{\frac{2}{r'},1\}})\)
convergence,
indicating that the bound for the pressure error possibly is not sharp for \(r>2\).

In \cite{Kaltenbach:2023b}, Kaltenbach/R\r{u}\v{z}i\v{c}ka obtain
first order convergence for the velocity-error of a local
discontinuous Galerkin method. The results, however are
\(\epsilon\)-dependent in the sense that the constants blow up fast as
\(\epsilon\to 0\). Moreover, they exploit additional regularity of
\(f\) and the involved jump penalizations result in non-monotone
schemes, with possible non-unique solutions; cf. \cite[\S
4]{Kaltenbach:2023a}.

In this paper we aim at achieving \eqref{intro:eq5} without the above
restrictions. To this end, we follow a different approach and extend
the quasi-optimal and pressure-robust approach
\cite{VerfuerthZanotti:2019,KreuzerZanotti:2020,KreuzerVerfuerthZanotti:2021}
of Kreuzer/Verf\"urth/Zanotti from linear to nonlinear Stokes
equations~\eqref{intro:eq1} with $(r,\epsilon)$-structure. In
particular, the methods have the following features: First, we propose
two discretizations with nonconforming Crouzeix-Raviart (CR) elements
that are stabilization-free and therefore yield monotone numerical
schemes. As a consequence, the numerical solution is unique.  Second,
a sophisticated \emph{smoothing} operator acting on the test functions
in the load term, allows to handle load-functions
$\ff\in\WW^{-1,r'}(\D)$ without 
requiring extra regularity of the data.  Third, this operator maps
element-wise divergence-free functions onto exactly
divergence-free functions, so to ensure the pressure robustness. Note
that the pressure robust methods of Linke et
al.~\cite{Linke:2014,Linke:2018} focus only on
the latter property and are not considered here. 

Based on these features, we prove a priori bounds of the form~\eqref{intro:eq5} (see
Theorems~\ref{cr:thm1} and \ref{T:aprioriCRDD1}). It is remarkable
that the results neither resort to additional 
regularity of the data and are robust for the critical limit
\(\epsilon\to 0\)
of the regularisation parameter
in~\eqref{intro:eq2}.

As a consequence, we obtain for regular velocities
$\F(\DB\uu)\in W^{1,2}(\D)^{d\times d}$ full first order convergence
independent of the pressure
\begin{align}\label{eq:BestRate}
  \norm{\F(\DB\uu)-\F(\DB_h\uu_h)}_2
  \lesssim \norm{h\nabla\F(\DB\uu)}_2\,.
\end{align}
As the pressure-error depends on the velocity-error, as a side effect,
the improved a priori bounds for the velocity also improve the bounds
for the pressure-error. In fact, if in addition \(p\in
W^{1,r'}(\Omega)\) then we obtain 
\(\norm{p-p_h}_{r'}=\mathcal{O}( h^{\min\{\frac{2}{r'},1\}})\), which
corresponds to the rates observed in the
numerical examples of \cite{Belenki:2012} for the MINI element using a
different method. 


The plan of the paper is as follows: In Section \ref{sec:prelim}, we 
introduce the \(\phi\)-structure (a generalization of the \((r,\ep)\)-structure), the nonlinear
problem~\eqref{intro:eq1} as well as the finite element
framework. Sections~\ref{sec:cr:modified_cr_discretization} and
~\ref{S:pr-cr-DB=DD} each presents the a priori analysis for the
velocity-error of a pressure robust method based on
Crouzeix-Raviart elements for \(\DB=\nabla\) and \(\DB=\DD\)
respectively. Section~\ref{sec:press-err} concerns the improved bounds
for the pressure-error followed by some numerical experiments in 
Section~\ref{sec:numerical_experiments}.

\section{Preliminaries}\label{sec:prelim}

To begin with, we clarify our notation and we state important
properties of the nonlinear operator $\A$ and of the natural distance.  
Further, we introduce the variational formulation of \eqref{intro:eq1}.

\subsection{Basic notation and function spaces}
Let $\R^+$ be the set of all positive real numbers, and $\R^+_0\coloneqq\R^+\cup\{0\}$. 
The Euclidean scalar-product of two vectors $\pp,\,\qq\in \R^d$ is denoted by
$\pp\cdot\qq$ and the Frobenius product of $\PP,\,\QQ\in\R^{d\times d}$
is defined by $\PP:\QQ\coloneqq\sum_{i,j=1}^dP_{ij}Q_{ij}$.  
We set \(|\pp|\coloneqq(\pp\cdot\pp)^{1/2}\) and $|\QQ|\coloneqq(\QQ:\QQ)^{1/2}$.

For an open set \(\omega\subset \R^d\), and $r\in [1,\infty]$, we denote by $L^r(\omega)$ the space of
scalar \(r\)-integrable functions on \(\omega\) with corresponding
norm $\norm{\cdot}_{r;\omega}$.
The space $L^r_0(\omega)$ is the
closed subspace of $L^r(\omega)$ of functions with
vanishing mean value.  
We equip the first order Sobolev space 
$W^{1,r}(\omega)$ 
with norm \(\|\cdot\|_{1,r;\omega}=(\|\cdot\|_{r;\omega}^r+\|\nabla
\cdot\|_{r;\omega}^r)^{1/r}\) and denote by
$W^{1,r}_0(\omega)$ its closed subspace of functions with vanishing
trace on $\partial\D$. For \(r\in(1,\infty)\) and $\frac{1}{r}+\frac{1}{r'}=1$, i.e. $r'=\tfrac{r}{r-1}$,  we have that
$W^{1,r}_0(\omega)$ is a reflexive Banach space with dual space
$W^{-1,r'}(\omega)=\big(W^{1,r}_0(\omega)\big)^*$
and for the dual pairing between $f\in W^{-1,r'}(\omega)$ and $v\in
W^{1,r}_0(\omega)$, we use the notation $\langle f,v\rangle_\omega$.  We
use the convention
\(W^{0,r}(\omega)=W_0^{0,r}(\omega)=L^r(\omega)\). Orlicz
  spaces $L^\phi(\omega)$ and Sobolev-Orlicz spaces
  $W^{1,\phi}(\omega)$ are introduced later in
  Section~\ref{sec:prop-nonl-oper}.

In~case of $\omega=\D$, we usually omit the index~$\D$, e.g. we write
$\norm{\cdot}_{r}$ instead of $\norm{\cdot}_{r;\D}$.  
Spaces of $\R^d$-valued functions are denoted with boldface type, though no distinction is made in the notation 
of norms and inner products; for instance, the norm in $\WW^{1,r}(\D)= (W^{1,r}(\D))^d$ is given by 
$\norm{\vv}_{1,r}=\left(\sum_{1\leq i\leq d}
  \norm{ v_i}_{1,r}^r\right)^{1/r}$. Thanks to Poincar{\'e} and Korn inequalities, an alternative norm on
$\WW^{1,r}_0(\D)$ is given by \(\|\DB \cdot\|_{r}\) if
\(r\in(1,\infty)\); compare e.g. with \cite[Theorem
6.10]{DieningRuzickaSchumacher:2010}. 

For a set \(\omega\subset \R^d\) with Hausdorff dimension \(d\) or
\(d-1\), we denote by \(|\omega|\) its $d$- or $(d-1)$-dimensional Hausdorff measure and for
\(f\in L^r(\omega),\) we use the
abbreviation 
$$
\langle f\rangle_\w\coloneqq\dashint_{\w}f\coloneqq\frac{1}{|\w|}\int_\w f
$$
for its mean value in \(\omega\).

In order to simplify the notation, we shall often denote \(a\lesssim b\) when
\(a\le cb\) for a constant \(c>0\) only depending on fixed but maybe
problem specific constants. For \(a\lesssim
b\lesssim a\) we shall also write \(a\eqsim b\).

\subsection{N-functions}
\label{sec:n-functions}

A convenient abstract framework for problems with an
\((r,\epsilon)\)-structure like~\eqref{intro:eq2} is based on so-called
N-functions, 
that are standard in the theory of Orlicz spaces;
cf. \cite{RaoRen:1991,KokilashviliKrbec:1991}. A continuous, convex,
and strictly monotone function $\psi:\R_0^+\to\R_0^+$ is
called an N-function if
\begin{align*}
  \psi(0)=0,\qquad \lim_{t\to 0}\frac{\psi(t)}{t}=0\qquad\text{and}\qquad
  \lim_{t\to\infty}\frac{\psi(t)}{t}=\infty.
\end{align*}
Thanks to the convexity of \(\psi\), there exists its right derivative $\psi'$, which is non-decreasing and satisfies $\psi'(0)=0$, $\psi'(t)>0$ for $t>0$ and $\lim_{t\to\infty}\psi'(t)=\infty$.

The conjugate of an N-function $\psi$ is defined by 
\begin{align*}
\psi^*(t)\coloneqq\sup_{s\geq 0}\left(st-\psi(s)\right)\quad \text{for all }t\geq 0
\end{align*}
and is again an N-function and we have \((\psi^*)^*=\psi\). As an
example: the conjugate of $\psi(t)=\frac 1r t^r$ is $\psi^*(t) =
\frac{1}{r'} t^{r'}$ with $\frac 1r + \frac 1{r'}=1$. 

For a quantitative numerical approach we require more regular
N-functions. We say that an N-function $\psi$ is uniform convex, if
additionally $\psi \in W^{2,1}_{\textrm{loc}}((0,\infty))$ (or
$C^1([0,\infty))$ and piecewise $C^2$) and 
\begin{align}
  \label{eq:uniformly-convex}
  1<r^- \coloneqq\inf_{t>0} \frac{\psi''(t)\,t}{\psi'(t)}+1
 \leq  \sup_{t>0} \frac{\psi''(t)\,t}{\psi'(t)}+1  \eqqcolon r^+<\infty.
\end{align}
We call $r^\pm$ the indices of uniform convexity of~$\psi$. In
the following we assume that $\psi$ is a uniformly convex
N-function. Note that uniform convexity rules out almost linear growth
as well as exponential growth. 

Let us recall the properties of uniformly convex N-functions
from~\cite[Appendix~B]{DieningFornasierTomasiWank:2020}. If $\psi$ is
uniformly convex with indices $r^-$ and $r^+$, then $\psi^*$ is
uniformly convex with indices $(r^+)'$ and $(r^-)'$ respectively
. 
Furthermore,
for all $s,t \geq 0$ there holds
\begin{align}
  \label{eq:phi-st}
  \min \set{s^{r^-},s^{r^+}} \,\psi(t) \leq \psi(st) \leq
  \max \set{s^{r^-},s^{r^+}} \,\psi(t).
\end{align}
As a consequence $\psi$ and $\psi^*$ satisfy the so-called
$\Delta_2$-condition with
\begin{align*}
  \psi(2t) \leq 
  2^{r^+}
  \psi(t) \quad
  \text{and}\quad
  \psi^*(2t)
  \leq 
                           2^{(r^-)'}
  \psi^*(t) ,\qquad &\text{for all $t\geq 0$}.
\end{align*}
By convexity, we have the quasi-triangle inequality
\begin{align}\label{eq:qtriangle}
  \psi(s+t) \leq \tfrac 12 \psi(2s) + \tfrac 12 \psi(2t) \leq 2^{r^+-1} \left(\psi(s)+\psi(t)\right)\quad\forall s,t\ge0.
\end{align}
It follows also from~\eqref{eq:phi-st} and $st \leq \psi(s)+\psi^*(t)$ that
\begin{align}\label{eq:deltaYoung}
  st\le \delta^{1-r^+} \psi(s)+\delta\psi^*(t),\qquad \text{ for all $s,t\ge0, \delta \in (0,1]$.}
\end{align}
Moreover, we have for all \(t\ge 0\) that
\begin{align}\label{eq:psi*psi'=psi}
  \psi(t)\eqsim t \psi'(t)\qquad\text{and}\qquad
  \psi^*(\psi'(t))\eqsim \psi(t)
\end{align}
with constants depending solely on $r^\pm$.

We next introduce the notion of shifted N-function first introduced in
\cite{Diening:2005}  and further developed for example in
\cite{Ruzicka:2006,Diening_2007,BelenkiDieningKreuzer:12,Kreuzer:2013,DieningFornasierTomasiWank:2020}. Here,
we use the version of
\cite[Appendix~B]{DieningFornasierTomasiWank:2020},
where also the self-contained proofs for the properties below
can be found.  

For a given uniformly convex N-function \(\psi\) with indices of uniform
convexity $r^\pm$, the family of shifted functions
$\{\psi_a\}_{a\geq 0}$ is defined by 
\begin{align}\label{def:shifted_phi_function}
  \psi_a(t)\coloneqq\int_0^t\psi_a'(s) \,\d s
  \qquad\text{ with }\qquad \psi_a'(t)\coloneqq\frac{\psi'(\max\set{a,t})}{\max \set{a,t}}t\,.
\end{align}
In the original work~\cite{Diening:2005}, the maximum $\max\set{a,t}$
was replaced by $a+t$. 
However, though equivalent, the
definition~\eqref{def:shifted_phi_function} has some technical
advantages. 
We have $(\psi_a)^* = (\psi^*)_{\psi'(a)}$ and the functions
$\psi_a$ are again uniformly convex and with indices of uniform
convexity $\min \set{r^-,2}$ and $\max \set{r^+,2}$.  Hence, $\psi_a$ and $(\psi_a)^*$
satisfy the $\Delta_2$-condition with constants independent of~$a\geq
0$. 

A first basic property states for all $\PP,\QQ\in\R^{d\times d}$ that
\begin{align}\label{eq:shift-symmetry}
  \psi_{|\PP|}(|\PP-\QQ|)\eqsim \psi_{|\QQ|}(|\PP-\QQ|)
  \quad\text{and}\quad\psi_{|\PP|}'(|\PP-\QQ|)\eqsim
  \psi_{|\QQ|}'(|\PP-\QQ|).
\end{align}
Moreover, for $t\geq 0$ and all $\PP,\QQ \in \R^{d \times d}$ we have
\begin{align}\label{eq:prime-shift-change}
  \bigabs{\psi'_{|\PP|}(t)-
    \psi'_{|\QQ|}(t)} & \lesssim \psi_{|\QQ|}'\big(\bigabs{\abs{\PP}-\abs{\QQ}}\big)\leq \psi_{|\QQ|}'(\abs{\PP-\QQ}),
\end{align}
where the hidden constant depends only on $r^\pm$.
Based on this, for any $\delta \in (0,1]$, all $t\geq 0$ and all $\PP,\QQ \in \R^{d \times d}$ there holds
\begin{subequations}\label{eq:change_of_shift}
  \begin{align}
    \psi_{|\PP|}(t)&\le
    C\, \delta^{1-r^+}\,\psi_{|\QQ|}(t)
    +\delta \abs{\F(\PP) - \F(\QQ)}^2\,,
    \label{eq:change_of_shift1}
    \\
    \left(\psi_{|\PP|}\right)^*(t)&\leq C\, \delta^{1-(r^-)'} \left(\psi_{|\QQ|}\right)^*(t)
    + \delta \abs{\F(\PP) - \F(\QQ)}^2\,,
    \label{eq:change_of_shift2}
  \end{align}
  where $C$ only depends on~$r^\pm$. These estimates are often called \emph{shift-change}.
\end{subequations}

When \(\DB=\DD\) we shall need the following Korn inequality on
bounded John domains, which can e.g. be found in~\cite[Theorem
6.13]{DieningRuzickaSchumacher:2010}. The precise definition of a John
domain is not important in this context, just note that it applies to
any 
bounded Lipschitz
domain. 
\begin{proposition}[{\cite[Theorem 6.13]{DieningRuzickaSchumacher:2010}}]
  \label{P:Korn1}
  Let $\psi$ be a uniformly convex N-function with indices~$r^\pm$. Then for a bounded John domain~$\omega\subset\R^d$ and all
  \(\ww\in\WW^{1,1}(\omega)\) we have 
  \begin{equation*}
    \begin{aligned}
      \int_\omega\psi\left(\left|\nabla\ww-\langle\nabla\ww\rangle_\omega\right|\right)
      &\lesssim \int_\omega \psi\left(\left|\DD\ww-\langle\DD\ww\rangle_\omega\right|\right)\,.
    \end{aligned} 
  \end{equation*}
  The hidden constant only depends on the constants $r^\pm$ of uniform
  convexity of \(\psi\) and the John constant
  of~$\omega$. Note that if $\ww \in \WW^{1,1}_0(\omega)$, then
  $\mean{\nabla \ww}_\omega=0$ and $\mean{\DD \ww}_\omega=0$. 
\end{proposition}

\subsection{Properties of the nonlinear operator}
\label{sec:prop-nonl-oper}

In this section we state our assumptions on the nonlinear operator
$\A$ from~\eqref{intro:eq1}. To this end, let \(\phi\) be a fixed uniformly convex N-function
with indices of uniform convexity $r^\pm$.  For simplicity, we 
assume that $\phi \in C^2((0,\infty))$ (although
$\phi \in W^{2,1}_{\textrm{loc}}((0,\infty))$ would suffice in what follows). 

For a bounded open set~$\omega \subset \R^d$, we denote by
$L^\phi(\omega)$ and $W^{1,\phi}(\omega)$ the usual Orlicz
and Sobolev-Orlicz spaces; see \cite{RaoRen:1991}. In fact,
\(v\in L^\phi(\omega)\) if \(\int_\omega\phi(|v|)<\infty \) and \(v\in
W^{1,\phi}(\omega)\) if in addition 
\(\int_\omega\phi(|\nabla v|)<\infty \). The space  $L^\phi(\omega)$
is a reflexive Banach space equipped with the Luxembourg norm
\(\|v\|_\phi:=\inf\{\lambda\colon \int_\omega\phi(|v|/\lambda)\le
1\}\). 
By
$L^\phi_0(\omega)$ we denote the closed subspace of~$L^\phi(\omega)$ of 
functions with integral zero. 
The subspace $W^{1,\phi}_0(\omega)$ of Sobolev-Orlicz functions
with zero boundary values is a reflexive Banach space with norm
\(\|\nabla\cdot\|_{\phi}\) and 
we denote its dual space by
$W^{-1,\phi^*}(\omega)$.

The most prominent example for $\phi$ is
\begin{align}\label{def:phi_function}
	\phi(t)\coloneqq\int_0^t \phi'(s)\,\d s \qquad\text{with }\phi'(t)\coloneqq(\ep+t)^{r-2}t\,
\end{align}
with $1 < r < \infty$ and $\epsilon \geq 0$ and we say that
$\phi$ has $(r,\epsilon)$-structure.
In this case, we have  $L^\phi(\omega) = L^r(\omega)$ and $W^{1,\phi}(\omega) =
W^{1,r}(\omega)$. 
If $\epsilon=0$, then $r^\pm=r$
and if $\epsilon>0$, then $r^- = \min \set{r,2}$ and $r^+= \max
\set{r,2}$. This follows from 
\begin{align*}
  \frac{\phi''(t)\,t}{\phi'(t)} +1&= \frac{rt+2\epsilon}{t+\epsilon}.
\end{align*}
In particular, the
indices of uniform convexity of~$\phi$ are bounded independently
of~$\epsilon \geq 0$. As a consequence, all estimates that only depend
on the indices of uniform convexity are independent of~$\epsilon \geq
0$.
\begin{remark}
  Other examples of uniformly convex N-functions are
  \begin{enumerate}
  \item $\phi(t) = t^{r_1} +  t^{r_2}$ for $1<r_1 \leq r_2 < \infty$.
  \item $\phi(t) = t^r \log(e+t)$ for $1<r<\infty$.
  \end{enumerate}
\end{remark}

\begin{assumption}[Nonlinear operator]\label{assumption:nonlinear_operator}
  We assume that the nonlinear operator $\A:\,\R^{d\times d}\rightarrow\R^{d\times d}$ belongs to
  $C^{0}(\R^{d\times d},\R^{d\times d})\cap C^1(\R^{d\times d}\setminus\{\bzero\},\R^{d\times d})$ 
  and satisfies $\A(\bzero)=\bzero$. 
  Furthermore, we assume that  $\A$ has $\phi$-structure, i.e., there exist
  constants $C_0,C_1>0$, such
  that\footnote{For functions $g:\,\R^{d\times d}\rightarrow\R$ we use the notation $\partial_{kl}g(\QQ)
  \coloneqq\frac{\partial g(\QQ)}{\partial Q_{kl}}$.}
	\begin{subequations}\label{assumption:nonlinear_operator:eq}
  \begin{align}
    \sum_{i,j,k,l=1}^d\partial_{kl}A_{ij}(\QQ)P_{ij}P_{kl}&\ge
    C_0 \phi''(\abs{\QQ})|\PP|^2,\label{assumption:nonlinear_operator:1}\\
    |\partial_{kl}A_{ij}(\QQ)|&\leq C_1\phi''(\abs{\QQ})\label{assumption:nonlinear_operator:2}
  \end{align}
	\end{subequations}
  holds for all $\PP,\,\QQ\in\R^{d\times d}$ with $\QQ\neq \bzero$ and all $i,j,k,l\in\{1,\ldots,d\}$.
\end{assumption}
\begin{remark}
  If $\phi$ is a uniformly convex N-function, then $\A:\,\R^{d\times d}\rightarrow\R^{d\times d}$ defined by
  \begin{align*}
    A(\QQ) = \phi'(\abs{\QQ}) \frac{\QQ}{\abs{\QQ}}
  \end{align*}
  has automatically $\phi$-structure.
\end{remark}
\begin{remark}
  If $\phi$ has $(r,\epsilon)$-structure, then
  \begin{align*}
    \sum_{i,j,k,l=1}^d\!\!\!\partial_{kl}A_{ij}(\QQ)P_{ij}P_{kl}\ge \widetilde{C}_0(\epsilon+\abs{\QQ})^{r-2}|\PP|^2,\qquad
    |\partial_{kl}A_{ij}(\QQ)|\leq \widetilde{C}_1(\epsilon+\abs{\QQ})^{r-2}\,
  \end{align*}
  and
  \begin{align*}
    \phi_a(t)\eqsim\left(\ep+a+t\right)^{r-2} t^2\quad\text{as well
    as }\quad(\phi_a)^*\eqsim \left( (\ep+a)^{r-1}+t\right)^{r'-2} t^2
  \end{align*}
  uniformly in $t,a \geq 0$ with constants only depending on~$r$ (but
  not on~$\epsilon$).
\end{remark}

Several studies on the finite element analysis of the $r$-Laplace
equation (e.g.,
\cite{BarrettLiu:1993,Barrett:1993,Barrett:1994,Ruzicka:2006,Belenki:2012})
indicate that, for quasi-optimal error bounds, a certain error measurement adapted
to the structure of \(\phi\) is necessary. In order to introduce this
so-called natural distance, we define 
\begin{align}\label{def_F}
  \F:\,\R^{d\times
  d}\rightarrow\R^{d\times d}\qquad\text{by}\qquad\F(\QQ)\coloneqq \sqrt{\phi'(|\QQ|)|\QQ|}\,\frac{\QQ}{|\QQ|}.
\end{align}
If $\phi$ has $(r,\epsilon)$-structure, then $ \F(\QQ)=\left(\ep+|\QQ|\right)^{\frac{r-2}{2}}\QQ$.

The functions $\A$ and $\F$ are closely related as depicted in following
Lemma from~\cite{Diening:2005,Diening_2007} and
\cite[Appendix~B]{DieningFornasierTomasiWank:2020}. 
\begin{lemma}
  \label{ZSH}
  Let $\A$ satisfy Assumption~\ref{assumption:nonlinear_operator} and
  let $\F$ be defined by \eqref{def_F}. Then for all $\PP,\,\QQ\in
  \R^{d\times d}$, we have
  \begin{align*}
    \big(\A(\PP)-\A(\QQ)\big) : (\PP-\QQ) 
    &\eqsim\left(\phi_{|\PP|}\right)^*(|\A(\PP)-\A(\QQ)|)
    \\
    &\eqsim |\F(\PP)-\F(\QQ)|^2
    \\
    &\eqsim\phi_{|\PP|}(|\PP-\QQ|)
    \\
    \intertext{and}
    |\A(\PP)-\A(\QQ)|&\eqsim \phi'_{|\PP|}(|\PP-\QQ|),
  \end{align*}
  where the hidden constants only depend on $r^\pm$.
\end{lemma}

According to Lemma \ref{ZSH}, for $v,u\in W^{1,\phi}(\D)$, we have the equivalence
\begin{align*}
  \int_\Omega(\A(\DB v)-\A(\DB u)):(\DB v-\DB u)
  \eqsim\norm{\F(\DB v)-\F(\DB u)}^2_2,
\end{align*}
where the constants only depend on~$r^\pm$.  We call the square root
of the right-hand side the natural distance or $\F$-distance.

\begin{remark}[Quasi-norm and natural distance]
  The natural distance is an important concept in the regularity
  theory as well as the numerical analysis of nonlinear PDEs. It was
  introduced under the name \emph{quasi-norm} for
  $(r,\epsilon)$-structure by Barrett and Liu in the breakthrough
  paper~\cite{Barrett:1994}. In particular, they used the quantity 
  \begin{align*}
    |w|_{(r, v)}^{2}\coloneqq\int_{\D}\big(\epsilon+|\nabla v|+|\nabla w|
    \big)^{r-2}|\nabla w|^{2}
    \,, \qquad  v, w\in W^{1,r}(\D).
  \end{align*}
  Indeed, in this case, we have for all \(\PP,\QQ\in\R^{d\times
    d}\) that
  \begin{align*}
    \phi_{|\PP|}(|\PP-\QQ|)
    &\eqsim (\ep+|\PP|+|\QQ|)^{r-2}|\PP-\QQ|^2
  \end{align*}
  which implies the equivalence to the natural distance, i.e.
  \begin{align}
    \int_\Omega(\A(\nabla v)-\A(\nabla u)):(\nabla v-\nabla u)
    \eqsim\norm{\F(\nabla v)-\F(\nabla u)}^2_2
    \eqsim \abs{ v- u}_{(r, v)}^2\label{connection-quasi-norm:eq1}.
  \end{align}
  For the quasi-norm error, Barrett and Liu proved quasi-optimality for 
  conforming finite element Galerkin approximation of the $p$-Laplace
  problem. This was extended in \cite{Ruzicka:2006} to more general
  problems  
  with uniformly convex N-functions.  
\end{remark}

We conclude this section stating equivalences of different integral
averages  in the natural distance.
\begin{lemma}[{\cite[Lemma A.2]{DieKapSchw:2012}}]\label{lem:anymean}
Let $\omega\subset\R^d$ be a bounded open set and let $\QQ\in L^\phi(\omega)^{d\times d}$. 
Further we define $\mean{\QQ}_\omega^{\A} \coloneqq \A^{-1}(\mean{\A(\QQ)}_\omega)$. 
Then we have
  \begin{align*}
    \dashint_\omega \abs{\F(\QQ) - \mean{\F(\QQ)}_\omega}^2 \eqsim \dashint_\omega
    \abs{\F(\QQ) - \F(\mean{\QQ}^{\A}_\omega)}^2 \eqsim \dashint_\omega
    \abs{\F(\QQ) - \F(\mean{\QQ}_\omega)}^2
  \end{align*}
  and the hidden constants only depend on 
  \(r^\pm\)
  .
\end{lemma}

\subsection{Variational formulation of the nonlinear Stokes equations}

Let $\phi$ be a uniformly convex N-function with indices~$r^\pm$ of
uniform convexity and
assume that  $\A$ has $\phi$-struture. We define the velocity
and pressure spaces by
\begin{align}\label{def_V_Q}
  \WW^{1,\phi}_0(\D)\qquad\text{and}\qquad L^{\phi^*}_0(\D).
\end{align}
The variational formulation of \eqref{intro:eq1} reads as: for $\ff\in \WW^{-1,\phi^*}(\D)$ 
find $\uu\in\WW^{1,\phi}_0(\D)$ and $p\in L_0^{\phi^*}(\Omega)$ with
\begin{subequations}\label{wp:eq1}
\begin{alignat}{2}
&\forall \vv\in
\WW^{1,\phi}_0(\D)&\qquad\int_{\D} \A(\DB\uu):\DB\vv 
-  \int_{\D}
p \divergenz\vv 
&= \langle\ff,\vv\rangle
,\label{wp:eq1a}\\
&\forall q\in
L^{\phi^*}_0(\D)&\qquad\int_{\D} q \divergenz \uu 
& =0  .\label{wp:eq1b}
\end{alignat}
\end{subequations}
Equation \eqref{wp:eq1b} shows that the velocity is actually in the
space of divergence-free vector fields
\begin{align}\label{def_Z}
\Z\coloneqq\{\vv\in \WW^{1,\phi}_0(\D) \mid \divergenz\vv=0\}\subset \WW^{1,\phi}_0(\D) \,.
\end{align}
Consequently, 
\(\uu\)
solves the reduced problem 
\begin{align}\label{wp:eq2}
  \forall \vv\in\Z\qquad\int_{\D} \A(\DB\uu):\DB\vv 
  = \langle\ff,\vv\rangle.
\end{align}
According to the strict  monotonicity of \(\A\) (cf. Lemma~\ref{ZSH}),
existence and uniqueness of~$\uu \in \Z$ follows by standard monotone operator
theory and Korn's inequality in Proposition~\ref{P:Korn1}.
The inf-sup condition \cite[Lemma~4.3]{Belenki:2012} for Orlicz spaces reads
  \begin{gather}
    \label{eq:inf-sup-norm}
    \inf_{q\in L_0^{\phi^*}(\Omega)}\sup_{\vv\in \WW^{1,\phi}_0(\D) }
    \frac{\int_\D q\divergenz
      \vv}{\norm{\nabla\vv}_{\phi}\norm{q}_{\phi^*}} =:\beta>0,
  \end{gather}
where $\beta>0$ only depends on~$r^\pm$ and~$\Omega$ via its John
constant. This inf-sup condition guarantees the existence of a unique
pressure~$p \in
L^{\phi^*}_0(\D)$ such that $(\uu,p) \in \WW^{1,\phi}_0(\D)
\times L^{\phi^*}_0(\D)$ solves~\eqref{wp:eq1}. 

\subsection{Mesh and finite element spaces}
\label{sec:fe-framework}

At this point we introduce our notation for the finite element
spaces of the following sections. 
Let $\T$ be a shape regular, face-to-face decomposition of $\D$ consisting of $d$-simplices $K\in\T$ 
such that  $\overline{\D}=\bigcup_{K\in\T}K$. 
The faces  of the elements in $\T$ are denoted by $\FC$. The subset of faces in the interior of
\(\Omega \) is denoted by \(\FCi\) and its complement of boundary
faces by \(\FCb\). The skeleton of \(\T\) is \(\Sigma\coloneqq\cup_{F\in\FC}F\).

For $K\in\T$ and \(F\in\FC\) we denote by $h_K$
respectively \(h_F\) its diameter and we introduce, with some ambiguity of
notation, the mesh-size function
\(h:\Omega\to\R^+\) by
\begin{align*}
  h_{|K}\coloneqq h_K, \quad \text{in}~K\in\T,\qquad \text{and}\qquad h_{|F}\coloneqq h_F,\quad\text{on}~F\in\FC
\end{align*}
and let \(h_{\max}=\|h\|_{L^{\infty}(\Omega)}\).
Denoting by $\rho_K$, \(K\in\T\), the supremum of the diameters of all balls
inscribed in the respective \(K\), the shape constant of \(\T\) is given by
\begin{equation}\label{mesh_1}
  \max_{K\in \T}\frac{h_{K}}{\rho_{K}}<\infty.
\end{equation} 
We emphasise that, in the following, constants hidden in \(\lesssim\) or
\(\eqsim\) may depend on the shape constant of \(\T\) but not on the mesh-size function.

Each face $F\in\FCi$ in the interior of $\D$ is oriented by  
a unit normal vector $\nn_F$. We indicate by $\Jump{ v}_F$ and \(\Avg{v}_F\) the jump respectively average of \(v\) on the face $F\in\FCi$ 
in direction of $\nn_F$. In particular, for
$K_1,K_2\in\T$ with $F=K_1\cap K_2$ we define for \(\xx\in F\)
\[
\Jump{ v}_{F}(\xx)=v_{|K_1}(\xx)-v_{|K_2}(\xx)\quad\text{and}\quad \Avg{v}_{F}(\xx)=\frac{v_{|K_1}(\xx)+v_{|K_2}(\xx)}{2},
\]
where $\nn_F$ points outside $K_1$. 
Of course, the sign of $\Jump{v}_F$ depends on the orientation of $\nn_F$, 
which however is not relevant to our presentation. For boundary faces $F\in\FCb$, 
$\nn_F$ is oriented outside $\D$ and 
$\Jump{v}_F$ equals the trace on $F$, i.e. for $K\in\T$ with $F=K\cap\partial\D$ 
we define $\Jump{v}_F(\xx)=v|_K(\xx)$ for $\xx\in
F$. We let \(\nn\in L^\infty(\Sigma)\) with \(\nn_{|F}=\nn_F\), \(F\in\FC\).

For a differential operator $\mathcal D$, the notation $\mathcal D_h$ is used for the broken version of $\mathcal D$, 
i.e.
\begin{align}\label{eq:DBh}
\left({\mathcal D}_h v\right)_{|K}\coloneqq{\mathcal D}\left(v_{|K}\right)\qquad\text{in}~K,~K\in\T
\end{align}
for piecewise smooth $v$. E.g., the broken gradient of a piecewise $W^{1,r}$-function $v$ is given by 
$\left(\nabla_h v\right)_{|K}=\nabla\left(v_{|K}\right)$ for all $K\in\T$. 
For an element $K\in \T$ we define the set of neighbours $\Neigh_{K}$ and the neighbourhood
$\omega_{K}$ by
\begin{subequations}\label{element_neighbourhood}
  \begin{equation}
    \begin{aligned}
      \Neigh_{K}\coloneqq\{K'\in \T \mid K'\cap K\neq\emptyset\},\qquad
      \omega_{K}\coloneqq\text{interior} \left(\bigcup_{K'\in
          \Neigh_{K}}K'\right)\,.
    \end{aligned}
  \end{equation}
  The sets $\omega_K$ are open, bounded and connected.  The
  shape constant~\eqref{mesh_1} of $\T$ implies the following
  properties of $\T$:
  \begin{align}
    |\omega_{K}|\eqsim |K|\eqsim h_K^d\quad \text{for all}~K\in \T \text{ and } \# \Neigh_{K}\leq m_{0} \text{ for some }
    m_{0}\in\N.
  \end{align}
\end{subequations}
For $\ell\geq 0$ let $\PS_\ell(K)$ and $\PS_\ell(F)$ be the set of polynomials 
of degree at most $\ell$ on a $d$-simplex $K\in\T$ and a face $F\in\FC$, respectively. 
For \(k\in \mathbb{N}_0\), the space of $W^{k,1}$-conforming element-wise polynomials of order
\(\ell\in\N_0\) on $\T$ is defined by 
\begin{align}\label{space_broken_polynomials:eq1}
S^k_\ell\coloneqq S^k_\ell(\T)\coloneqq \left\{v\in W^{k,1}(\D)\mid \forall K\in\T \ v_{|K}\in\PS_\ell(K)\right\}.
\end{align}

The lowest-order Crouzeix-Raviart finite element is given as the subspace of
functions in \(S_1^0\), for which the face mean values of jumps
vanish, i.e. 
\begin{align}\label{CR_space:eq1}
  \CR\coloneqq\CR(\T)\coloneqq\left\{v\in S^0_1 \mid
    \forall F\in\FC 
    \int_F \Jump{v} =0
  \right\}.
\end{align}
Note that \(\int_F \Jump{v}_F=\int_F v\) for boundary
edges \(F\in\FCb\).
In other words, $\CR(\T)$ consists of all functions in $S^0_1(\T)$
that are continuous at the barycenters  
of interior faces and vanish at the barycenters of boundary faces.

It is well known that broken Sobolev norms are definite on \(\CR\) and
they even dominate jumps across edges. As we could find this statement
only for Hilbert norms, we provide a short proof in the required
\(W^{1,1}\) setting. 
\begin{lemma}\label{L:jump<nabla}
  Let $v\in\CR(\T)+ W^{1,1}_0(\Omega)$, then
  for all $F\in\FC$ we have
  \begin{align*}
    \dashint_F h_F^{-1}\left|\Jump{v}_F \right|
    \lesssim
    \sum_{\substack{K\in\T\\ F\subset K}}
    \dashint_{K}\left|\nabla v\right|
    \,.
  \end{align*}
  The hidden constant depends only on the shape constant of \(\T\).
\end{lemma}

\begin{proof}
  Since $v\in\CR(\T)+ W^{1,1}_0(\Omega)$, the face mean value
  $\mean{v}=\dashint_F v$ takes the same value on both
  sides of the face is therefore well defined. If $F \subset \partial
  \Omega$, then $\mean{v}_F=0$. 

  We first consider interior faces \(F\in\FCi\), i.e. there exist
  $K_1,K_2\in\T$ with $F=K_1\cap K_2$.  Then the embedding
  $W^{1,1}(K_i) \hookrightarrow L^1(F)$, \(i=1,2\), and a Poincar{\'e} inequality
  imply 
  \begin{align}
    \begin{aligned}
      \dashint_F \left|\Jump{v}_F \right|
      &\leq \sum_{i=1}^2 \dashint_F \abs{v_{|K_i} -\mean{v}_F}
      \leq 2 \sum_{i=1}^2 \dashint_F \abs{v_{|K_i} -\mean{v}_{K_i}}
      \lesssim \sum_{i=1}^2 \dashint_{K_i} \abs{\nabla v}
      .
    \end{aligned}
    \label{cr:symgrad:lem1:proof:eq1}
  \end{align}
  The case $F \subset \partial \Omega$ is similar using $\mean{v}_F=0$.
\end{proof}

The \emph{Crouzeix-Raviart} interpolation operator
\(\ICR:W^{1,1}(\D)\to\CR\) is defined by
\begin{align*}
\int_F\ICR v=\int_F v\quad \forall F\in\FC^i.
\end{align*}
Element-wise integration by parts yields for all $K\in\T$ and $v\in W^{1,1}(\D)$ 
\begin{align}\label{interpolation_operator_cr:eq1}
\nablah(\ICR v)_{|K}=\dashint_K \nablah(\ICR v)
  =\dashint_K \nabla_h v
\end{align}
i.e., $\ICR$ preserves averages of first derivatives.

Using Jensen's inequality and the fact
that \(\DB_h(\ICR \vv)\) is piecewise constant, we conclude
from~\eqref{interpolation_operator_cr:eq1} that for any
N-function~$\psi$ we have
\begin{align}\label{eq:ICRstab}
	\int_K \psi(|\DB_h(\ICR \vv)|) 
	=  \int_K \psi\Big(\Big|\dashint_K \DB
	\vv(\yy)\,\d\yy\Big|\Big)
	\le \int_K \psi(|\DB\vv|).
\end{align}
Thanks to the definition of the Luxembourg norm, this readily implies
\(\|\DB_h \ICR \vv\|_{\psi}\le \|\DB\vv\|_{\psi}\), i.e.  
\(\ICR\) is \(\WW^{1,\psi}\)-stable
with constant one. 
Moreover, we even have that \(\ICR\vv\) is a
quasi-optimal approximation of~$\vv$ in the
natural distance.
\begin{corollary}\label{C:QOICR}
	For any \(\vv\in \WW_0^{1,\phi}(\D)\) and $K \in \mathcal{T}$, we have
	\begin{align*}
		\|\F(\DB \vv)-\F(\DB_h\ICR\vv)\|_{2;K}\lesssim
    \inf_{\QQ_K\in\R^{d\times d}}\|\F(\DB \vv)-\QQ_K\|_{2;K}.
	\end{align*}
	The hidden constant depends only on \(r^{\pm}\).
\end{corollary}
\begin{proof}
	Note that \(\DB_h(\ICR \vv)_{|K}=\dashint_K
	\DB \vv\) due to \eqref{interpolation_operator_cr:eq1} and use
        Lemma~\ref{lem:anymean}.
\end{proof}

With some ambiguity of notation, we extend all the terms introduced
above to vector-valued functions by applying them component by
component and possibly adding them up, e.g. in the case of norms.

Following \cite{CrouzeixRaviart:1973}, we discretize \eqref{wp:eq1} with nonconforming Crouzeix-Raviart functions for the velocity and 
piecewise constant functions for the pressure
\begin{align}\label{def_CR_S0}
  \CRd\coloneqq(\CR)^d\quad\text{and}\quad 
  {\wh S}_0^0\coloneqq S_0^0(\T)\cap L^{\phi^*}_0(\D).
\end{align}
This pair is inf-sup stable. Indeed, the interpolant $\ICR$ is a bounded Fortin operator, according to \eqref{interpolation_operator_cr:eq1} and \eqref{eq:ICRstab}. This observation and \eqref{eq:inf-sup-norm} imply
    \begin{gather}
  \label{CR:inf_sup_cond:norm}
  \inf_{q_h\in {\wh S}_0^0}\sup_{\vv_h\in\CRd} 
  \frac{\int_\Omega q_h\divergenz_h \vv_h}{\norm{\nablah\vv_h}_{\phi}\norm{q_h}_{\phi^*}} \ge\beta>0.
\end{gather}

\subsection{Smoothing operator}
\label{sec:smoother}

Note that the Crouzeix-Raviart space is nonconforming, i.e.
$\CRd\not\subset \WW^{1,\phi}_0(\D)$ in view of the lack of global continuity 
and the violation of the boundary conditions.
Consequently the datum $\ff\in
\WW^{-1,\phi^*}(\D)$ in \eqref{wp:eq1} cannot be directly applied on
test-functions from \(\CRd\). Still, it is known from
\cite{VeeserZanotti:2018} that an error estimate like
\eqref{intro:eq5}, not involving extra regularity of $\uu$, can hold
true only for discretizations defined for all possible data
$\ff$. Therefore, we resort to a so-called \textit{smoothing operator}
$\EE: 
\CRd \to \WW_0^{1,\infty}(\D)\subset \WW_0^{1,\phi}(\D)$, which acts on the test-functions, thus
enabling the application of $\ff$.  

We make use of the smoothing operator introduced in \cite[\S4.1]{VerfuerthZanotti:2019} for the Stokes equations (see also \cite[\S3]{KreuzerVerfuerthZanotti:2021}). The operator acts on $\vv_{h} \in \CRd$ as follows
\begin{equation*}
\label{E:smoothing-operator}
\EE \vv_h = \A \vv_h + \BB \vv_h + \CC \vv_h
\end{equation*}
where  
\begin{subequations}\label{E:Cons+Entire}
	\begin{itemize}
		\item $\A: \CRd \to (S_1^1)^{d}$ is an averaging operator, enabling the stability estimates \eqref{Eq:E_CR} below;
		\item $\BB$ maps $\vv_h$ to a combination of face bubbles, so as to enforce the identity 
		\begin{align}
			\int_F \EE\vv_h = \int_F \vv_h\qquad\forall
			F\in\FC\label{E:Edge-Moments}
		\end{align}
		which implies
		\begin{align}\label{E:nabladG=nablaE}
			\int_\Omega \QQ_h:\nablah\vv_h
			= \int_\Omega \QQ_h:\nabla \EE\vv_h \qquad \forall \QQ_h\in (S_0^0)^{d\times d},
		\end{align}
		in combination with element-wise integration by parts;
		\item $\CC$ maps $\vv_h$ to a combination of element bubbles, so as to enforce the identity
		\begin{align}\label{E:Divh=DivE}
			\divergenz \EE\vv_h=\divh \vv_h.
		\end{align}
	\end{itemize}
\end{subequations}
We refer to \cite{VerfuerthZanotti:2019,KreuzerVerfuerthZanotti:2021} for the details of the construction and the proof of \eqref{E:Cons+Entire}.

\begin{remark}[Pressure robustness]
  \label{R:pressure-robustness}
  Owing to \eqref{E:Divh=DivE}, we readily infer
  \begin{equation*}
    \divh \vv_h = 0 \quad \implies \quad \divergenz \EE\vv_h = 0
  \end{equation*}
  i.e. $\EE$ maps element-wise divergence-free functions to exactly divergence-free functions. This property, originally pointed out in \cite{Linke:2014}, guarantees the pressure robustness of the discretizations proposed in the next sections, meaning that the discrete velocity is independent of the discrete pressure. This property is necessary to make sure that the velocity-error is independent of the pressure error, thus preventing from sub-optimal error decay rates as in \eqref{intro:eq4}. 
\end{remark}

\begin{lemma}[Stability of $\EE$]\label{L:E_CR}
The operator  $\EE: \CRd\to \WW_0^{1,\infty}(\D)$
from~\cite{VerfuerthZanotti:2019,KreuzerVerfuerthZanotti:2021} 
satisfies~\eqref{E:Cons+Entire} and the local stability
properties 
\begin{subequations}\label{Eq:E_CR}
  \begin{alignat}{2}
    \norm{\nabla \EE\vv_h}_{1;K}&\lesssim \norm{\nablah \vv_h}_{1;\omega_K}&\quad\forall
    K&\in\T,\label{Eq:E_CRa}
    \\
    \|\DB_h (\ww_h+\EE\vv_h)\|_{\infty;K}&\lesssim \dashint_{K}
    |\DB_h (\ww_h+\EE\vv_h)|
    &\quad\forall K&\in\T,
    \label{Eq:E_CRb}\\
    \norm{\DB_h(\vv_h-\EE\vv_h)}_{1;K}
    &\lesssim \sum_{\substack{F\in\FC\\ F\cap K\neq\emptyset}}
    \int_F \left|\Jump{ \vv_h}_F \right|
    &\quad\forall K&\in\T,\label{Eq:E_CRc}
\end{alignat}
for \(\ww_h,\vv_h\in \CRd\) and \(\DB\in \{\nabla,\DD\}\). The hidden constants depend only on the shape constant of $\T$.
\end{subequations}
\end{lemma}

\begin{proof}Properties~\eqref{E:Cons+Entire} are proved
  in~\cite[\S4.1]{VerfuerthZanotti:2019}.
  
  The bound~\eqref{Eq:E_CRc} is proved in~\cite[proposition
  18]{KreuzerVerfuerthZanotti:2021} for \(L^2\)-type norms and
  \(\DB=\nabla\). Recalling that \(|\DD (\vv_h-\EE\vv_h)|\le |\nabla
  (\vv_h-\EE\vv_h)|\),
  the claimed
  bound follows from scaling. 

  Assertion~\eqref{Eq:E_CRa} follows from a triangle inequality
  \begin{align*}
    \norm{\nabla \EE\vv_h}_{1;K}\le \norm{\nabla (\EE\vv_h-\vv_h)}_{1;K}+\norm{\nabla \vv_h}_{1;K}
  \end{align*}
  with~\eqref{Eq:E_CRc} and Lemma~\ref{L:jump<nabla}.

  Since by the above indicated construction, \(\EE\) maps to an affine
  equivalent finite element, therefore property~\eqref{Eq:E_CRb}
  follows from equivalence of norms on finite dimensional spaces and
  standard scaling of affine equivalent elements.
\end{proof}
The local $\WW^{1,1}$-stability of $\EE$~\eqref{Eq:E_CRa} implies local $\WW^{1,\phi}$-stability in the following sense. 
  
  \begin{lemma}[Local $W^{1,\psi}$-stability of $\EE$]\label{cr:lem1}
    Let $\psi$ be a uniformly convex N-function. Then $\EE: \CRd\to \WW_0^{1,\infty}(\D)$ from~\cite{VerfuerthZanotti:2019,KreuzerVerfuerthZanotti:2021} satisfies
    \begin{align}\label{cr:lem1:eq1}
      \dashint_K\psi\left(|\nabla \EE\vv_h|\right)
      \lesssim \dashint_{\omega_K} \psi\left(|\nablah \vv_h|\right)
      \qquad \forall K \in \T.
    \end{align}
    The hidden constant only depends on the indices of uniform convexity
    of \(\psi\) and on the shape constant of \(\T\).
  \end{lemma}
  
  \begin{proof}
	Using~\eqref{Eq:E_CRb}, the local $\WW^{1,1}$-stability~\eqref{Eq:E_CRa}  of $\EE$ 
	together with \eqref{element_neighbourhood} and the \(\Delta_2\)-condition, we estimate for all $K\in\T$ and $\vv_h\in\CRd$
	\begin{align}\label{eq:JensenE}
		\dashint_K\psi\left(|\nabla \EE\vv_h|\right)
		\lesssim \dashint_K\psi\left(\dashint_K|\nabla \EE\vv_h|\right)
		\lesssim \dashint_K\psi\left(\dashint_{\omega_K}|\nablah \vv_h|
		\right) \lesssim \dashint_{\omega_K}\psi\left(|\nablah \vv_h|\right).
	\end{align}
        Here we have used Jensen's inequality for convex functions in
        the last step. 
\end{proof}

Lemma \ref{cr:lem1} implies in particular, that $\EE$ is 
$\WW^{1,\phi}$-stable.  This and the fact that only finite many of
the \(\omega_K\), \(K\in\T\), overlap imply for all \(\vv_h\in\CRd\)
that 
\begin{align}\label{cr:lem1:eq2}
\int_\Omega\phi(|\nabla\EE\vv_h|)\lesssim \int_\Omega\phi(|\nablah\vv_h|)	\quad\text{resp.}\quad\norm{\nabla\EE\vv_h}_{\phi}\lesssim\norm{\nablah\vv_h}_{\phi}.
\end{align}

Finally, we observe that \(\EE\) is a right inverse of the
Crouzeix-Raviart interpolation operator as a consequence of \eqref{E:Edge-Moments}, cf. \cite[Lemma~3.2]{VeeserZanotti:2019}.
\begin{corollary}\label{C:rightInverse}
  For all $K\in\T$ and $\vv_h\in \CRd(\T)$ we have
  \begin{align*}
    \ICR\EE \vv_h=\vv_h\qquad\text{and}\qquad \dashint_{K} \nabla
    \EE \vv_h
    =\dashint_{K} \nabla \vv_h
    =\nablah\vv_{h|K}\quad\forall K\in\T.
  \end{align*}
\end{corollary}

\section{A pressure robust Crouzeix-Raviart method for $\DB=\nabla$.}
\label{sec:cr:modified_cr_discretization}
In this section we propose a pressure robust Crouzeix-Raviart method for the
nonlinear Stokes equation~\eqref{intro:eq1} for the case $\DB=\nabla$.
The method coincides with the ones
in~\cite{VerfuerthZanotti:2019,KreuzerVerfuerthZanotti:2021} for
$\phi(t)=t^2$,  i.e. for the linear Stokes equations.

\begin{method}\label{M:DD=nabla}
  For $\ff\in\WW^{-1,\phi^*}(\D)$ compute $\uu_h\in \CRd$ and
  $p_h\in {\wh S}_0^0$ such that
    \begin{alignat*}{2}
      \forall \vv_h&\in\CRd&\qquad\int_{\D} \A(\nablah
      \uu_h):\nablah\vv_h
      - \int_{\D} p_h \divergenz_h\vv_h
      &= \langle \ff, \EE \vv_h\rangle
      ,
      \\
      \forall q_h&\in{\wh S}_0^0&\quad\int_{\D} q_h \divergenz_h
      \uu_h
      & =0 .
    \end{alignat*}
\end{method}
Note that we do not require regularity beyond $\ff\in
\WW^{-1,\phi^*}(\D)$ from \eqref{wp:eq1}, because \(\EE\vv_h\in \WW_0^{1,\infty}(\D)\).
Similarly to the derivation of \eqref{wp:eq2}, 
we have that \(\uu_h\) is in the space of element-wise divergence-free functions
\begin{align}\label{def_Zh}
\Z_h\coloneqq\{\vv_h\in \CRd \mid \divergenz_h\vv_h=0\}\subset\CRd\,,
\end{align}
and uniquely solves the reduced problem 
\begin{align}\label{cr:eq3}
\forall \vv_h\in\Z_h\quad\int_{\D} \A(\nablah\uu_h):\nablah\vv_h 
  = \langle\ff,\EE\vv_h\rangle.
\end{align}

In view of the stability \eqref{cr:lem1:eq2} 
and the inf-sup-condition \eqref{CR:inf_sup_cond:norm}, the
well-posedness of Method~\ref{M:DD=nabla} follows from standard monotone
operator theory; cf. \cite{Belenki:2012,Hirn:2010}.

\begin{proposition}[Well-posedness]\label{P:M-DB=nabla}
  Method~\ref{M:DD=nabla} determines a unique pair $(\uu_h,p_h)\in\CRd\times
  {\wh S}_0^0$ depending continuously on the datum \(\ff\in \WW^{-1,\phi^*}(\Omega)\).
\end{proposition}

From the properties of $\EE$ and the $\varphi$-structure of $\A$ 
we infer our first main result addressing the approximation
properties of Method~\ref{M:DD=nabla}.

\begin{theorem}[Velocity-error]\label{cr:thm1}
Let $(\uu,p)\in\WW_0^{1,\phi}(\D)\times L_0^{\phi^*}(\D)$ be the solution to \eqref{wp:eq1}  
and let $(\uu_h,p_h)\in\CRd\times {\wh S}_0^0$ be computed by Method~\ref{M:DD=nabla}. 
Then we have 
\begin{align}\label{cr:thm1:eq1}
\norm{\F(\nabla\uu)-\F(\nablah\uu_h)}_{2}^2
\lesssim \sum_{K\in\T} \inf_{\QQ_K\in\R^{d\times d}}\norm{\F(\nabla\uu)-\QQ_K}_{2;\omega_K}^2\,.
\end{align}
In particular, when $\F(\nabla \uu)\in\WW^{1,2}(\D)$, we have
\begin{align}\label{cr:thm1:eq2}
\norm{\F(\nabla\uu)-\F(\nablah\uu_h)}_2\lesssim \norm{h\nabla\F(\nabla\uu)}_2\,.
\end{align}
The hidden constants depend only on $\Omega$, 
  \(r^\pm\)
  ,  and the shape constant of $\T$. 
\end{theorem}

Before we prove Theorem~\ref{cr:thm1}, some remarks are in order.

  \begin{remark}[Quasi-optimality]\label{R:qo-discuss}
    Recall that, for conforming and divergence-free discretizations,
    the velocity-error estimate \eqref{intro:eq5} is derived from the
    quasi-optimality stated in
    \eqref{intro:eq3}. Theorem~\ref{cr:thm1} establishes a counterpart
    of \eqref{intro:eq5} for Method~\ref{M:DD=nabla}, but
    quasi-optimality in the sense of \eqref{intro:eq3} is not
    guaranteed. The critical aspect is the stability of the smoothing
    operator $\EE$, as Lemma~\ref{cr:lem1} for 
    \(\psi=\phi_a\),  requires a 
    constant shift \(a\ge0\) on element neighbourhoods. Enforcing this through a  change
    of the shift~\eqref{eq:change_of_shift} introduces a suboptimal
    approximation term. 
    For details consider the bound of $I_2$ in the proof of Theorem~\ref{cr:thm1} below. 
\end{remark}

\begin{remark}[Rate of convergence and regularity]
\label{R:comparisonV}
We emphasize that two aspects of the velocity-error estimate
\eqref{cr:thm1:eq2} improve upon  the existing results in
\cite{Belenki:2012} for the MINI element and in
\cite{Hirn:2010} for a stabilized $\mathds{Q}_1/\mathds{Q}_1$
discretization. First, the convergence rate for the velocity is optimal, whereas for
problems with \((r,\ep)\)-structure, 
suboptimal rates are observed in
\cite{Belenki:2012,Hirn:2010} both theoretically and in numerical simulations
for $r>2$, cf. \eqref{intro:eq4}.
Second, no additional regularity of the pressure is
necessary beyond the minimal requirement $p \in L^{\phi^*}_0(\D)$ from
\eqref{wp:eq1}. 
Both improvements are made possible
by the pressure robustness observed in
Remark~\ref{R:pressure-robustness}. 

An \(\mathcal{O}(h)\) velocity-error estimate can be found
in~\cite[Corollary 4.2]{Kaltenbach:2023b} for a local dG
discretization, however, compared to~\eqref{cr:thm1:eq2} this result
has several disadvantages. First, the discrete solution is obtained
via a pseudo-monotone scheme, which does not guarantee uniqueness of
the solution. Second, the constant hidden in \(\mathcal{O}(h)\) is not
robust with respect to $\varepsilon$ and is unbounded as \(\epsilon\to
0\), and third, the additional regularity \(\ff\in\LL^2(\Omega)\) is required. In
contrast, Theorem~\ref{cr:thm1} invokes only the minimal regularity
\(\F(\nabla \uu) \in W^{1,2}(\Omega)^{d\times d}\) that is necessary
according to the approximation theory.

The regularity for system~\eqref{wp:eq1} has not been fully
developed. However, it is expected that this kind of regularity is the
natural one. 
Results for the special case of $(r,\epsilon)$-structure can be found
in~\cite{VeigaKaplickyRuzicka:2011,Veiga_2013,Veiga:2008}. The case
without pressure but for all uniformly convex N-functions is
considered in~\cite[Theorem~2.4]{BehnDiening2024}.  
\end{remark}

\begin{proof}[Proof of Theorem~\ref{cr:thm1}]
  We abbreviate $\vv_h=\ICR\uu\in\CRd$. We use Lemma~\ref{ZSH} to estimate the error by
  \begin{align*}
\norm{\F(\nabla\uu)-\F(\nabla_h\uu_h)}_2^2 
  &\eqsim \int_{\D}\left(\A(\nabla_h
    \uu_h)-\A(\nabla\uu)\right):(\nabla_h\uu_h-\nabla
  \uu)
  \\
  &=\underbrace{\int_{\D}\left(\A(\nabla_h
    \uu_h)-\A(\nabla\uu)\right):(\nabla_h\vv_h-\nabla
  \uu)}_{=:I_1}
  \\
  &\quad+\underbrace{\int_{\D}\left(\A(\nabla_h
    \uu_h)-\A(\nabla\uu)\right):(\nabla_h\uu_h-\nabla_h
    \vv_h)}_{=:I_2}.
  \end{align*}
We bound the first term with the Young-type inequality
\eqref{eq:deltaYoung} and Lemma~\ref{ZSH} by
\begin{align*}
  I_1&\le \int_{\D} \delta\left(\phi_{|\nabla\uu|}\right)^*\left(\left|
  \A(\nabla_h\uu_h)-\A(\nabla\uu)  \right|\right) + \int_{\D}C_\delta
  \,\phi_{|\nabla \uu|}\left(\left| \nabla_h\vv_h-\nabla\uu
  \right|\right)
  \\
  &\lesssim \delta\,\|\F(\nabla\uu)-\F(\nabla_h\uu_h)\|_{2}^2+C_\delta
    \|\F(\nabla\uu)-\F(\nabla_h\vv_h)\|_{2}^2. 
\end{align*}
For estimating the second term on the right-hand side, observe that the second equation of Method~\ref{M:DD=nabla} and~\eqref{interpolation_operator_cr:eq1} imply
\(\divergenz_h(\uu_h - \vv_h)_{|K}=-\dashint_K \divergenz\uu =0\) for all $K \in \T$. This and \eqref{E:Divh=DivE} reveal $\divergenz \EE(\uu_h-\vv_h)=0$. Hence, the first equation of
Method~\ref{M:DD=nabla} and~\eqref{wp:eq1} imply
\begin{align*}
I_2&=\langle\ff,\EE(\uu_h-\vv_h)\rangle - \int_{\D}\A(\nabla\uu):(\nabla_h\uu_h-\nabla_h
    \vv_h)
  \\
  &=\int_{\D}\A(\nabla\uu):\left(\nabla
    \EE(\uu_h-\vv_h)-\nablah(\uu_h-\vv_h)\right)
  \\
   &=
     \int_{\D}\left(\A(\nabla\uu)-\A(\nablah\vv_h)\right):\left(\nabla
    \EE(\uu_h-\vv_h)-\nablah(\uu_h-\vv_h)\right), 
\end{align*}
where we have used~\eqref{E:nabladG=nablaE} together with \(\A(\nablah\vv_h)\in
(S_0^0)^{d\times d}\) in the last step.

By the Young-type inequality  \eqref{eq:deltaYoung}, the
quasi-triangle inequality \eqref{eq:qtriangle}, and Lemma~\ref{ZSH}, we conclude 
for $\delta>0$ that
\begin{align*}
I_2&\le
2\int_{\D} C_{\delta}\left(\phi_{|\nabla\uu|}\right)^*\left(\left| \A(\nabla\uu)-\A(\nablah\vv_h)  \right|\right)
  \\
  &\quad+ \int_{\D}\delta\,\phi_{|\nabla\uu|}\left(\left|
    \nablah(\uu_h-\EE\uu_h) \right|\right)
    + \int_{\D}\delta\,\phi_{|\nabla\uu|}\left(\left| \nablah (\vv_h-\EE\vv_h) \right|\right)
  \\
&\eqsim C_{\delta}\norm{\F(\nabla\uu)-\F(\nablah\vv_h)}^2_2  
  \\
  &\quad+ \int_{\D}\delta\,\phi_{|\nabla\uu|}\left(\left|
    \nablah(\uu_h-\EE\uu_h) \right|\right)
    + \int_{\D}\delta\,\phi_{|\nabla\uu|}\left(\left| \nablah (\vv_h-\EE\vv_h) \right|\right).
\end{align*}
The first term on the right-hand side is already fine and the latter
terms can be treated in the same way. 
For  $\QQ_K\in\R^{d\times d}$ we first change the shift with \eqref{eq:change_of_shift} from
$|\nabla\uu|$ to $|\QQ_K|$ to obtain
\begin{align*}
 \int_{K}\phi_{|\nabla\uu|}\left(\left| \nablah (\vv_h-\EE\vv_h)
   \right|\right)
   \lesssim \int_{K}\phi_{|\QQ_K|}\left(\left| \nablah
       (\vv_h-\EE\vv_h) \right|\right)
   +\int_{K}\abs{\F(\nabla \uu) -\F(\QQ_K)}^2.
\end{align*}
Next, Lemma~\ref{L:E_CR} and
Lemma~\ref{L:jump<nabla} yield that
\begin{align*}
  \| \nablah (\vv_h-\EE\vv_h) \|_{\infty;K}&\lesssim \dashint_K 
  \left|\nablah (\vv_h-\EE\vv_h) \right|
    =\sum_{\substack{F\in\FC\\ F\cap K\neq\emptyset}}\dashint_F h^{-1}\left|\Jump{ \vv_h-\uu }_F
  \right|
  \\
                                           &\lesssim  \sum_{\substack{F\in\FC\\ F\cap K\neq\emptyset}}\sum_{\substack{K'\in\T\\ F\subset K'}}
  \dashint_{K'}\left|\nablah( \vv_h-\uu)\right|
  \eqsim \dashint_{\omega_K} \left|\nablah( \vv_h-\uu)\right|. 
\end{align*}
The last estimate follows from
\begin{align*}
\omega_K=\bigcup\{K'\in\T: F\subset
K'~\text{for some}~F\in\FC~\text{with}~F\cap K\neq\emptyset\}
\end{align*}
and the fact that every element \(K'\subset \omega_K\) appears in previous
sum at least once and at most \((d+1)\)-times. Consequently, recalling
\(|K|\eqsim|\omega_K|\), we obtain with Jensen's
inequality and a shift-change back to $|\nabla\uu|$, that
\begin{align}\label{eq:NatDist(uh-Euh)}
  \begin{aligned}
  \int_{K}\phi_{|\nabla\uu|}&\left(\left| \nablah (\vv_h-\EE\vv_h) \right|\right)
  \\
    &\lesssim \int_{\omega_K}\phi_{|\QQ_K|}\left(\left| \nablah
        (\vv_h-\uu) \right|\right)
    +\int_{K}\abs{\F(\nabla \uu) -\F(\QQ_K)}^2
    \\
    &\lesssim \int_{\omega_K} \phi_{|\nabla \uu|}\left(\left| \nablah
        (\vv_h-\uu) \right|\right)
    +\int_{\omega_K} \abs{\F(\nabla \uu) -\F(\QQ_K)}^2
    \\
    &\eqsim \int_{\omega_K} \abs{\F(\nablah
        \vv_h)-\F(\nabla \uu)}^2
    +\int_{\omega_K} \abs{\F(\nabla \uu) -\F(\QQ_K)}^2.
  \end{aligned}
\end{align}
Estimating the term \(\int_{K}\phi_{|\nabla\uu|}\left(\left| \nablah
    (\uu_h-\EE\uu_h)\right|\right)\) in exactly the same way, and summing over all
    \(K\in\T\), we arrive at 
    \begin{align*}
      I_2&\lesssim C_\delta \norm{\F(\nabla\uu)-\F(\nablah\vv_h)}^2_2
      \\
      &\quad+ \delta\int_{\D} \abs{\F(\nabla_h
       \vv_h)-\F(\nabla \uu)}^2 +\delta\int_{\D} \abs{\F( \nablah
        \uu_h) - \F(\nabla \uu)}^2
      \\
      &\quad+ \delta\sum_{K\in\T}\inf_{\QQ_K\in\R^{d\times d}}\int_{\omega_K}\phi_{|\nabla\uu|}\left(\left| \nabla \uu - \QQ_K
     \right|\right)
    \end{align*}
    since \(\QQ_K\in\R^{d\times d}\) was arbitrary and the sets
    \(\omega_K\), \(K\in\T\), overlap only finitely often.

Using again Lemma~\ref{ZSH}, combining the previous estimates, we arrive at
\begin{align*}
  \norm{\F(\nabla\uu)-\F(\nablah\uu_h)}_2^2&\lesssim  (C_{\delta}+\delta)\norm{\F(\nabla\uu)-\F(\nablah\vv_h)}^2_2
  \\
  &\quad+\delta \norm{\F(\nabla\uu)-\F(\nablah\uu_h)}^2_2
  \\
  &\quad+\delta\sum_{K\in\T}\inf_{\QQ_K\in\R^{d\times d}}\norm{\F(\nabla\uu) - \F(\QQ_K)}_{2;\omega_K}^2\,.
\end{align*}
Choosing \(\delta\) sufficiently small, the second term can be
compensated on the left hand side and the assertion follows from
Corollary~\ref{C:QOICR}
recalling \(\vv_h=\ICR\uu\).
\end{proof}

\begin{remark}
  Note that we also have~\eqref{cr:thm1:eq1} for the post-processed
  approximation \(\EE\uu_h\) instead of \(\uu_h\). Indeed, thanks to
  Lemma~\ref{ZSH} and the quasi triangle
  inequality~\eqref{eq:qtriangle}, we have 
  \begin{align*}
    \|\F(\nabla \uu)-\F(\nabla\EE \uu_h)\|_2^2 \lesssim
    \|\F(\nabla \uu)-\F(\nablah \uu_h)\|_2^2+ \int_\Omega\phi_{|\nabla
    \uu|}(|\nablah
    \uu_h-\nabla\EE \uu_h|). 
  \end{align*}
  The claim follows estimating the former term on the right hand side by~\eqref{cr:thm1:eq1}
  and the latter term as in the proof of
  Theorem~\ref{cr:thm1}; compare with \eqref{eq:NatDist(uh-Euh)}.
\end{remark}

\section{A pressure robust Crouzeix-Raviart method for $\DB=\DD$}
\label{S:pr-cr-DB=DD}

Recall the symmetric gradient
\(\DD\vv\coloneqq\frac12(\nabla\vv+\nabla\vv^T)\). According
to~\cite{Arnold:1993}, on some choices of \(\Omega\) and 
\(\T\), there exists \(\zero\neq\vv_h\in \CRd\) with
\(\DD_h\vv_h=\zero\), 
where \(\DD_h\) is the element-wise symmetric gradient; compare with~\eqref{eq:DBh}.

This observation means that we cannot just apply Method~\ref{M:DD=nabla} with
\(\nablah\) replaced by \(\DD_h\) since the quasi-linear operator is
not strictly monotone in general. We propose the following alternative, inspired by the so-called recovered finite elements~\cite{GeorgoulisPryer:18}.

\begin{method}\label{M:DB=DD1}
  For $\ff\in \WW^{-1,\phi^*}(\D)$
  find $\uu_h\in \CRd$ and $p_h\in {\wh S}_0^0$ such that
  \begin{alignat*}{2}
      \forall \vv_h&\in\CRd\qquad\int_{\D} \A(\DD\EE \uu_h):\DD\EE\vv_h 
      &- \int_{\D} p_h \divergenz_h\vv_h
      &
      = \langle \ff, \EE \vv_h\rangle 
      \\
      \forall q_h&\in{\wh S}_0^0&\quad
      \int_{\D} q_h \divergenz_h \uu_h
      & =0\,.
  \end{alignat*}
\end{method}
\begin{remark}[Equivalent conforming method]
	\label{R:conforming-method}
	Method~~\ref{M:DB=DD1} is equivalent to a conforming and divergence-free discretization of \eqref{wp:eq1}: find $\widetilde{\uu}_h \in \EE(\CRd)$ and $\widetilde{p}_h \in {\wh S}_0^0$ such that
	\begin{subequations}\label{M:DB=DD1-equiv}
		\begin{alignat*}{2}
			&\forall \widetilde \vv_h\in
			\EE (\CRd)&\qquad\int_{\D} \A(\DB\widetilde \uu_h):\DB\widetilde \vv_h 
			-  \int_{\D}
			\widetilde p_h \divergenz\widetilde \vv_h 
			&= \langle\ff,\widetilde\vv_h\rangle
			,\\
			&\forall \widetilde q\in
			{\wh S}_0^0&\qquad\int_{\D} \widetilde q \divergenz \widetilde\uu_h 
			& =0 .
		\end{alignat*}
	\end{subequations}
	Indeed, we have $\widetilde \uu_h = \EE \uu_h$ and $\widetilde
        p_h = p_h$. The equivalence is obtained
        by~\eqref{E:Divh=DivE}. 
        Though the velocity
	space \(\EE(\CRd)\)  appears to have a complicated structure, it
	can be parametrised by \(\CRd\) via
	\(\EE\), which is one-to-one according to the first part of Corollary~\ref{C:rightInverse}.
\end{remark}

In analogy with \eqref{wp:eq2} and \eqref{cr:eq3}, we can characterize the velocity $\uu_h$ from Method~\ref{M:DB=DD1} as the unique solution in $\Z_h$ of the reduced problem 
\begin{align}
\label{cr:Red-DB=DD1}
\forall \vv_h\in \Z_h\quad\int_{\D} \A(\DD\EE\uu_h):\DD\EE\vv_h 
  = \langle\ff,\EE\vv_h\rangle
\end{align}
which is equivalent to a conforming Galerkin discretization of
\eqref{wp:eq2}, cf. Remark~\ref{R:conforming-method}. Consequently,
Method~\ref{M:DB=DD1} is uniquely solvable by 
standard monotone operator theory and the inf-sup-condition
\eqref{CR:inf_sup_cond:norm}. For an a priori bound of the velocity
$\uu_h$, we additionally recall the stability \eqref{cr:lem1:eq2} of
$\EE$ as well as the inverse bound 
\begin{align*}
	\int_{\D}\phi(|\nablah\vv_h|)=\int_{\D}\phi(|\nablah \ICR\EE\vv_h|)\leq \int_{\D}\phi(|\nabla \EE\vv_h|)\lesssim
	\int_{\D}\phi(|\DD \EE\vv_h|)\quad \forall\vv_h\in\CRd.
\end{align*}
resulting from Corollary~\ref{C:rightInverse}, the stability of the
Crouzeix-Raviart interpolation operator~\eqref{eq:ICRstab} and Korn's
inequality from Proposition~\ref{P:Korn1}.

\begin{proposition}[Well-posedness]\label{P:M-DB=DD1}
  Method~\ref{M:DB=DD1} determines a unique pair $(\uu_h,p_h)\in\CRd\times
  {\wh S}_0^0$, which depends continuously on \(\ff\in\WW^{-1,\phi^*}(\D)\).
\end{proposition}

The equivalence of Method~\ref{M:DB=DD1} with a conforming and divergence-free discretization implies also that we have quasi-optimality in the sense of \eqref{intro:eq3}.

\begin{theorem}[Approximation by $\EE \uu_h$]\label{T:M-DB=DD1}
Let $(\uu,p)\in\WW_0^{1,\phi}(\D)\times L_0^{\phi^*}(\D)$ be the solution of~\eqref{wp:eq1}  
and let $(\uu_h,p_h)\in\CRd\times {\wh S}_0^0$ be computed by Method~\ref{M:DB=DD1}. 
Then 
  \begin{equation*}
      \norm{\F(\DD\uu)-\F(\DD\EE\uu_h)}_{2} \lesssim
      \inf_{\vv_h\in \Z_h}\norm{\F(\DD\uu)-\F(\DD\EE\vv_h)}_{2}.
  \end{equation*}
  The hidden constant depends only on \(r^\pm\).
\end{theorem}

\begin{proof}
  Let \(\vv_h\in\Z_h\) be arbitrary. We employ Lemma~\ref{ZSH}, the scaled
  Young inequality~\eqref{eq:deltaYoung} and the Galerkin orthogonality of
  Method~\ref{M:DB=DD1} to
  obtain that
  \begin{align*}
    \norm{\F(\DD\uu)-\F(\DD\EE\uu_h)}_{2}^2 &\lesssim \int_\D
    \left(\A(\DD\uu)-\A(\DD\EE\uu_h)\right):\left(\DD\uu-\DD
    \EE\vv_h\right)
    \\
    &\lesssim \int_\D
    \delta\left(\phi_{|\DD
      \uu|}\right)^*\left(|\A(\DD\uu)-\A(\DD\EE\uu_h)|\right)
    \\
    &\qquad+\int_\D C_\delta\,\phi_{|\DD \uu|}\left(|\DD\uu-\DD
      \EE\vv_h|\right)
    \\
                                                 &\eqsim \delta \norm{\F(\DD\uu)-\F(\DD\EE\uu_h)}_{2}^2
    \\
    &\qquad+C_\delta \norm{\F(\DD\uu)-\F(\DD\EE\vv_h)}_{2}^2.
  \end{align*}
  The claimed bound follows for \(\delta>0\) small enough.
\end{proof}

If we consider the approximation of $\uu$ by $\uu_h$ instead of $\EE
\uu_h$, we have the following a priori error bounds. Note that, in this
case, quasi-optimality is not guaranteed. 

\begin{theorem}[Approximation by $\uu_h$]\label{T:M-DB=DD2}
	Let   $(\uu,p)\in\WW_0^{1,\phi}(\D)\times L_0^{\phi^*}(\D)$ be the solution of~\eqref{wp:eq1}  
	and let $(\uu_h,p_h)\in\CRd\times {\wh S}_0^0$ be computed by Method~\ref{M:DB=DD1}. 
	Then 
	\begin{align*}
		\norm{\F(\DD\uu)-\F(\DD_h\uu_h)}_{2}&\lesssim
		\inf_{\vv_h\in \Z_h}\big\{\norm{\F(\DD\uu)-\F(\DD\EE\vv_h)}_{2}
		+\norm{\F(\DD\uu)-\F(\DD_h\vv_h)}_{2}\big\}\,.
	\end{align*}
  The hidden constant depends only on 
   \(r^\pm\)
  , $d$ and the shape
constant of \(\T\).
\end{theorem}
\begin{proof}
  Let \(\vv_h\in\Z_h\) be arbitrary. Lemma~\ref{ZSH}, Corollary~\ref{C:rightInverse} and  \eqref{eq:ICRstab} reveal
  \begin{align*}
    \norm{\F(\DD_h\uu_h)-\F(\DD_h\vv_h)}_{2}^2&\eqsim
                                                   \int_{\D}\phi_{|\DD_h\uu_h|}\left(\left|\DD_h\ICR\EE(\uu_h-\vv_h)\right|\right)
    \\
    &\le  \sum_{K\in\T}\int_K\phi_{|\DD_h\uu_h|}\left(\dashint_K
      \left|\DD\EE(\uu_h-\vv_h)\right|\right)
    \\
    &\le \int_{\D}\phi_{|\DD_h\uu_h|}\left(
      \left|\DD\EE(\uu_h-\vv_h)\right|\right), 
  \end{align*}
  with Jensen's inequality in the last step. Changing the
  shift to \(|\DD\uu|\) with~\eqref{eq:change_of_shift} for some
  \(\delta>0\) and 
  applying the quasi-triangle inequality~\eqref{eq:qtriangle} yields
  \begin{multline*}
    \int_{\D}\phi_{|\DD_h\uu_h|}\left(
      \left|\DD\EE(\uu_h-\vv_h)\right|\right)\\
    \begin{aligned}
      &\le C_\delta\int_{\D}\phi_{|\DD\uu|}\left(
        \left|\DD\EE(\uu_h-\vv_h)\right|\right)
      +\delta \int_{\D} \abs{\F(\DD \uu) - \F(\DD_h \uu_h)}_2^2
      \\
      &\lesssim C_\delta\int_{\D}\phi_{|\DD\uu|}\left(
        \left|\DD\uu-\DD\EE\uu_h\right|\right)+C_\delta\int_{\D}\phi_{|\DD\uu|}\left(
        \left|\DD\uu-\DD\EE\vv_h\right|\right)
      \\
      &\qquad+\delta\, \norm{\F(\DD \uu) - \F(\DD_h \uu_h)}_2^2
      \\
      &\eqsim
      C_\delta\big(\|\F(\DD\uu)-\F(\DD\EE\uu_h)\|_2^2+\|\F(\DD\uu)-\F(\DD\EE\vv_h)\|_2^2\big)
      \\
      &\quad+ \delta \|\F(\DD\uu)-\F(\DD_h\uu_h)\|_2^2
      \\
      &\lesssim C_\delta\|\F(\DD\uu)-\F(\DD\EE\vv_h)\|_2^2+ \delta \|\F(\DD\uu)-\F(\DD_h\uu_h)\|_2^2,
    \end{aligned}
  \end{multline*}
  where we have used Theorem~\ref{T:M-DB=DD1} in the last step. With a quasi-triangle
  inequality, we thus obtain
  \begin{align*}
    \norm{\F(\DD\uu)-\F(\DD_h\uu_h)}_{2}^2&\lesssim
    \norm{\F(\DD\uu)-\F(\DD_h\vv_h)}_{2}^2+\norm{\F(\DD_h\uu_h)-\F(\DD_h\vv_h)}_{2}^2
    \\
    &\lesssim
      \norm{\F(\DD\uu)-\F(\DD_h\vv_h)}_{2}^2+C_\delta\norm{\F(\DD\uu)-\F(\DD\EE\vv_h)}_{2}^2
    \\
    &\quad+ \delta \|\F(\DD\uu)-\F(\DD_h\uu_h)\|_2^2.
  \end{align*}
  From this the claimed bound follows for \(\delta>0\) small enough.
\end{proof}

In order to establish velocity-error estimates like \eqref{intro:eq5} and \eqref{eq:BestRate}, we must assess the approximation power of the space $\EE (\CRd)$ in the natural distance.

\begin{lemma}[Smoothing error]\label{cr:symgrad:thm1}
  For all $\vv_h\in\CRd$  we have 
  \begin{align*}
    \norm{\F(\DD_h\vv_h)-\F(\DD\EE\vv_h)}_2^2
    \lesssim \sum_{K\in \T}\sum_{\substack{F\in\FC\\ F\cap K\neq\emptyset} } 
    \int_F h\, \phi_{|\DD_h\vv_{h|K}|}\left( h^{-1}\left|\Jump{ \vv_h}_F \right|\right)
    =:\J({\vv_h})\,.
  \end{align*}
  The hidden constant depends only on 
  \(r^\pm\)
  , \(d\) and the shape
	constant of \(\T\). 
\end{lemma}

\begin{proof}
Property~\eqref{Eq:E_CRc} implies for all
$\vv_h\in\CRd$ and $K\in\T$ that
\begin{align}
  \begin{aligned}
    \dashint_K \left| \DD_h\vv_h-\DD\EE\vv_h\right| 
    &\lesssim \sum_{\substack{F\in\FC\\ K\cap
        F\neq\emptyset}} \dashint_F h^{-1}\left|\Jump{\vv_h}_F \right|
  \end{aligned}
\label{cr:symgrad:thm4:proof:eq5}
\end{align}
Using Lemmas~\ref{ZSH} and~\ref{L:E_CR} and \eqref{cr:symgrad:thm4:proof:eq5}, we estimate 
\begin{align*}
\norm{\F(\DD_h\vv_h)-\F(\DD\EE\vv_h)}_{2;K}^2
&\eqsim \int_K
                                                  \phi_{|\DD_h\vv_h|}\left(|\DD_h\vv_h-\DD\EE\vv_h|\right)
  \\
&\lesssim\int_K \phi_{|\DD_h\vv_h|}\left(\dashint_K|\DD_h\vv_h-\DD\EE\vv_h|
       \right)
  \\
  &\lesssim\int_K \phi_{|\DD_h\vv_h|}\Bigg(\sum_{\substack{F\in\FC\\ F\cap K\neq\emptyset}} 
\dashint_F h^{-1}\left|\Jump{ \vv_h }_F \right|
  \Bigg)
  \,.
\end{align*}
It follows from the quasi-triangle inequality~\eqref{eq:qtriangle} and Jensen's-inequality that
\begin{align*}
\norm{\F(\DD_h\vv_h)-\F(\DD\EE\vv_h)}_{2;K}^2
&\lesssim\int_K \sum_{\substack{F\in\FC\\ F\cap K\neq\emptyset}} \phi_{|\DD_h\vv_h|}\left( 
\dashint_F h^{-1}\left|\Jump{ \vv_h }_F \right|
  \right)
  \\
&\lesssim\int_K \sum_{\substack{F\in\FC\\ F\cap K\neq\emptyset}} 
\dashint_F  \phi_{|\DD_h\vv_{h|K}|}\left( h^{-1}\left|\Jump{
       \vv_h }_F \right|\right)
       \\
&\eqsim
\sum_{\substack{F\in\FC\\ F\cap K\neq\emptyset} } 
\int_F h\, \phi_{|\DD_h\vv_{h|K}|}\left( h^{-1}\left|\Jump{\vv_h }_F \right|\right)
.
\end{align*}
By summing over all $K\in\T$, we arrive at the desired estimate.
\end{proof}

Motivated by the above result, we next investigate
the approximation in \(\CRd\) in the natural distance augmented with jumps.
\begin{lemma}\label{cr:symgrad:thm3}
For $\vv\in\WW_0^{1,\phi}(\D)$,
we have 
\begin{align*}
\norm{\F(\DD\vv)-\F(\DD_h\ICR\vv)}_{2}^2+ \J({\ICR\vv})
\lesssim \sum_{K\in\T}\inf_{\QQ_K\in\R^{d\times d}}\norm{\F(\DD\vv)-\F(\QQ_K)}_{2;\omega_K}^2\,.
\end{align*}
  The hidden constant depends only on 
  \(r^\pm\)
  , $d$ and the shape
	constant of \(\T\).
\end{lemma}

\begin{proof}
Owing to Corollary~\ref{C:QOICR}, we have for all \(K\in\T\) that
\begin{align}\label{eq:qo-ICR}
  \norm{\F(\DD\vv)-\F(\DD_h\ICR\vv)}_{2;K}^2\lesssim \inf_{\QQ_K\in\R^{d\times d}}\norm{\F(\DD\vv)-\F(\QQ_K)}_{2;K}^2.
\end{align}

In order to bound the term \(\J({\ICR\vv})\), we first employ~\eqref{eq:change_of_shift} to
change the shift to some arbitrary \(\QQ\in\R^{d\times d}\) and get
\begin{align}\label{eq:jump-chd-shift}
  \begin{aligned}
    &\sum_{\substack{K\in\T\\F\cap K\neq\emptyset}}\int_{F}h\phi_{\left|\DD\ICR\vv_{|K}\right|}
    \left(h^{-1}\left|\Jump{ \ICR\vv
        }_F\right|\right)
    \\
    &\qquad\lesssim \sum_{\substack{K\in\T\\F\cap K\neq\emptyset}}\int_{F}h\phi_{\left|\QQ\right|}
    \left(h^{-1}\left|\Jump{ \ICR\vv
        }_F\right|\right)
    +\sum_{\substack{K\in\T\\F\cap K\neq\emptyset}}\int_{K}\abs{\F(\DD_h \ICR\vv) 
        -\F(\QQ)}^2
    ,
     \end{aligned}
  \end{align}
where we have used that \(\DD_h \ICR\vv\) is piecewise constant. Using again Corollary~\ref{C:QOICR}, The latter term can be bounded by
\begin{align}\label{eq:jump-chd-shift<qo}
  \begin{aligned}
    \sum_{\substack{K\in\T\\F\cap K\neq\emptyset}}\|\F(\DD_h \ICR
    \vv)-\F(\QQ)\|_{2;K}^2
    \lesssim \sum_{\substack{K\in\T\\F\cap K\neq\emptyset}}\|\F(\DD
    \vv)-\F(\QQ)\|_{2;K}^2.
  \end{aligned}
\end{align}
For the
first term at the right hand side of~\eqref{eq:jump-chd-shift}, we have 
\begin{align*}
  \int_{F}h\phi_{\left|\QQ\right|}
\left(h^{-1}\left|\Jump{ \ICR\vv }_F\right|\right)
  \lesssim |F|h_F\,\phi_{\left|\QQ\right|}
\left(\dashint_F h^{-1}\left|\Jump{ \ICR\vv
  }_F\right|
  \right),
\end{align*}
where we used that \(\left|\Jump{\ICR\vv
 }_F\right|\lesssim \dashint_F\left|\Jump{ \ICR\vv
  }_F\right|
\) since \(\Jump{\ICR\vv
}_F\) is linear on \(F\). Note that thanks
to \(\vv \in \WW_0^{1,\phi}(\Omega)\), we have \(\Jump{ \vv}_{F}=0\) in
\(F\in\FC\).  Consequently, \(\Jump{ \ICR\vv}_{F}=\Jump{\ICR\vv-\vv}_{F}\) in
\(F\in\FC\) and it follows from
Lemma~\ref{L:jump<nabla} that
\begin{align*}
&\dashint_F h^{-1}\left|\Jump{\ICR\vv }_F \right|
=
\dashint_F h^{-1}\left|\Jump{ \ICR\vv - \vv}_F \right|
\lesssim\sum_{\substack{K\in\T\\F\subset K}}
\dashint_{K} \left|\nabla(\ICR\vv - \vv)\right|
  \,.
\end{align*}
This and Jensen's inequality imply
  \begin{align*}
    \int_{F}h\phi_{\left|\QQ\right|}
\left(h^{-1}\left|\Jump{ \ICR\vv }_F\right|\right)
&\lesssim \int_{F}h\phi_{\left|\QQ\right|}
\left(\sum_{\substack{K\in\T\\F\subset K}}
  \dashint_{K} \left|\nabla(\ICR\vv - \vv)\right|
    \right)
\\
&\lesssim \sum_{\substack{K\in\T\\F\subset K}}
  \int_{K}\phi_{\left|\QQ\right|}
  \left(\left|\nabla(\ICR\vv - \vv)\right|\right)
    \,.
  \end{align*}
  Thanks to~\eqref{interpolation_operator_cr:eq1} we can apply the
  Korn inequality (Proposition~\ref{P:Korn1})
  and obtain 
\begin{align*}
    \int_{F}h\phi_{\left|\QQ\right|}
\left(h^{-1}\left|\Jump{ \ICR\vv }_F\right|\right)
&\lesssim \sum_{\substack{K\in\T\\F\subset K}}
  \int_{K}\phi_{\left|\QQ\right|}
  \left(\left|\DD(\ICR\vv - \vv)\right|\right)
  \,.
\end{align*}
 At this point, we again invoke~\eqref{eq:change_of_shift} to change
 the shift to \(\DD\vv\) and obtain with Lemma~\ref{ZSH}
 \begin{align*}
   &\int_{F}h\phi_{\left|\QQ\right|}
   \left(h^{-1}\left|\Jump{ \ICR\vv }_F\right|\right)\\
     &\qquad \qquad\lesssim \sum_{\substack{K\in\T\\F\subset K}}\left(\int_{K}\phi_{\left|\DD\vv\right|}
     \left(\left|\DD(\ICR\vv -
         \vv)\right|\right)+\int_K\abs{\F(\DD\vv)-\F(\QQ)}^2 \right)
     \,
     \\
     &\qquad \qquad \lesssim \sum_{\substack{K\in\T\\F\subset K}}
     \left(\|\F(\DD\vv)-\F(\DD\ICR\vv)\|_{2;K}^2+\|\F(\DD\vv)-\F(\QQ)\|_{2;K}^2\right)
     \\
     &\qquad \qquad \lesssim \sum_{\substack{K\in\T\\F\subset K}}\|\F(\DD\vv)-\F(\QQ)\|_{2;K}^2\,,
 \end{align*}
 where we used~\eqref{eq:qo-ICR} in the last step. Combining this
 with~\eqref{eq:jump-chd-shift} and \eqref{eq:jump-chd-shift<qo} and
 recalling that \(\QQ\in\R^{d\times d}\) was arbitrary, we
 have proved for \(F\in \FC\) that
 \begin{align*}
    \sum_{\substack{K\in\T\\F\cap K\neq\emptyset}}\int_{F}h\phi_{\left|\DD\ICR\vv_{|K}\right|}
    \left(h^{-1}\left|\Jump{ \ICR\vv
  }_F\right|\right)
   \lesssim \sum_{\substack{K\in\T\\F\cap K\neq\emptyset}}\|\F(\DD\vv)-\F(\QQ)\|_{2;K}^2.
 \end{align*}
 Note that \(\F:\R^{d\times d}\to\R^{d\times d}\) is surjective and
 since \(\QQ\in\R^{d\times d}\) is arbitrary, we have
 \begin{align*}
   \sum_{\substack{K\in\T\\F\cap K\neq\emptyset}}\int_{F}h\phi_{\left|\DD\ICR\vv_{|K}\right|}
    \left(h^{-1}\left|\Jump{\ICR\vv}_F\right|\right)
    \lesssim \inf_{\QQ\in\R^{\d\times d}}\sum_{\substack{K\in\T\\F\cap K\neq\emptyset}}\|\F(\DD\vv)-\QQ\|_{2;K}^2.
 \end{align*}
 Together with~\eqref{eq:qo-ICR}, summing over all \(F\in\FC\) 
 finally proves
 the assertion. 
\end{proof}

With this preparation, we can prove the following main result. 
\begin{theorem}[Velocity-error]\label{T:aprioriCRDD1}
  Let $(\uu,p)\in\WW_0^{1,\phi}(\D)\times L_0^{\phi^*}(\D)$ be the solution to \eqref{wp:eq1}  
and let $(\uu_h,p_h)\in\CRd\times {\wh S}_0^0$ be computed by Method~\ref{M:DB=DD1}. 
Then we have 
  \begin{align*}
      \norm{\F(\DD\uu)-\F(\DD_h\uu_h)}_{2}^2+\norm{\F(\DD\uu)-\F(\DD\EE\uu_h)}_{2}^2 
      \lesssim \sum_{K\in\T}
                                                \inf_{\QQ_K\in\R^{d\times
                                                  d}}\norm{\F(\DD\uu)-\QQ_K}_{2;\omega_K}^2.
  \end{align*}
In particular, we have for  $\F(\DD\uu)\in\WW^{1,2}(\D)$ the estimate
\begin{align*}
\norm{\F(\DD\uu)-\F(\DD_h\uu_h)}_{2} + \norm{\F(\DD\uu)-\F(\DD\EE\uu_h)}_{2}
  \lesssim \norm{h\nabla\F(\DD\uu)}_{2}\,.
\end{align*}
  The hidden constant depends only on 
  \(r^\pm\)
  , $d$ and the shape
	constant of \(\T\). 
\end{theorem}

\begin{proof}
  We conclude from Theorems~\ref{T:M-DB=DD1} and~\ref{T:M-DB=DD2}, and
  Lemma~\ref{cr:symgrad:thm1} that
  \begin{multline*}
    \norm{\F(\DD\uu)-\F(\DD_h\uu_h)}_{2}^2+\norm{\F(\DD\uu)-\F(\DD\EE\uu_h)}_{2}^2 
    \\
    \lesssim \inf_{\vv_h\in\Z_h}\big\{\norm{\F(\DD\uu)-\F(\DD_h\vv_h)}_{2}^2+\J(\vv_h)\big\}.
  \end{multline*}
  Choosing \(\vv_h=\ICR\uu\in\Z_h\), the assertion follows from Lemma~\ref{cr:symgrad:thm3}. 
\end{proof}

\begin{remark}[Stabilization for the symmetric gradient problem]
In order to handle the non-definiteness of the
symmetric gradient, one can also consider a jump penalization of our 
method as it is common for discontinuous Galerkin (DG) methods; compare with \cite{DiPietroErn:2012}. 
This has been done by Kaltenbach and \Ruzicka{} in \cite{Kaltenbach:2023a,Kaltenbach:2023b,Kaltenbach:2023c} 
for a DG approximation of $(r,\epsilon)$-structure systems. 
However, their approach suffers from non-monotonicity of the stabilized nonlinear operator 
so that uniqueness of the discrete solution cannot be guaranteed. 
In contrast, our Method \ref{M:DB=DD1} features monotonicity and hence unique solvability. 
\end{remark}

\section{Estimates of the pressure-error} 
\label{sec:press-err}

Although the above velocity-error estimates for
Methods~\ref{M:DD=nabla} and 
\ref{M:DB=DD1} are pressure robust,
i.e. independent of the pressure error, the pressure-error itself
depends on the velocity-error. Consequently, improving the velocity
approximation can improve also the pressure approximation in some
cases.  To make this more precise, we combine ideas
 from~\cite[Theorem 4.7]{Belenki:2012} with the results in the
 previous sections.
For the ease of presentation, we restrict ourselves to problems with \((r,\epsilon)\)-structure.

\begin{theorem}[Pressure-error]\label{T:pressErr}
  Let \(\A\) have \((r,\epsilon)\)-structure \eqref{def:phi_function} 
  for some $1 < r < \infty$ and $\epsilon \geq 0$. 
  Let $(\uu,p)\in\WW_0^{1,r}(\D)\times L_0^{r'}(\D)$ be the unique solution to~\eqref{wp:eq1}  
  and let $(\uu_h,p_h)\in\CRd\times {\wh S}_0^0$ be obtained with either
  Method~\ref{M:DD=nabla} or 
\ref{M:DB=DD1}. Then we have 
\begin{align*}
\norm{p-p_h}_{r'}\lesssim  \left(\sum_{K\in\T}\inf_{\QQ_K\in\R^{d\times d}}
\norm{\F(\DB\uu)-\QQ_K}_{2;\omega_K}^2\right)^{\frac12\min\{1,\frac{2}{r'}\}}+\inf_{q_h\in {\wh S}_0^0}\norm{p-q_h}_{r'}\,.
\end{align*}
In particular, when $\F(\nabla \uu)\in\WW^{1,2}(\D)$ and \(p\in
W^{1,r'}(\Omega)\), we have
\begin{align}\label{eq:presRate}
	\norm{p-p_h}_{r'}\lesssim 
	\norm{h\nabla\F(\DB\uu)}_{2}^{\min\{1,\frac{2}{r'}\}}+\norm{h\nabla p}_{r'}.
\end{align}
Here \(\DB=\nabla\) or \(\DB=\DD\) depending on the context and
for \(r>2\), the hidden constant depends also on
\(\ff\) but is robust for \(\epsilon\to 0\).
\end{theorem}

\begin{proof}
  We  focus on the case of Method~\ref{M:DD=nabla}.  
  For Method~\ref{M:DB=DD1} the assertion follows with minor changes.
  Let $\pi p\in {\wh S}_0^0$ be the $L^2$-orthogonal projection of $p$ and apply the triangle inequality 
  \[
    \norm{p-p_h}_{r'}\leq \norm{p-\pi p}_{r'}+\norm{\pi p-p_h}_{r'}.
  \]
  Owing to Jensen's inequality, the projection $\pi$ is bounded,
  cf. \eqref{eq:ICRstab} where the same argument is used. Therefore,
  the first term on the right-hand side above is bounded by 
  \begin{equation*}
  	\norm{p-\pi p}_{r'} \leq 2 \inf_{q_h \in {\wh S}_0^0}\norm{ p - q_h}_{r'}.
  \end{equation*} 
  For the second term we invoke the inf-sup stability
  \eqref{CR:inf_sup_cond:norm}, $\divergenz_h \CRd = {\wh
    S}_0^0$ and \eqref{E:Divh=DivE}, to obtain
  \begin{align*}
        \beta\norm{\pi p-p_h}_{r'}&\le
                              \sup_{\vv_h\in\CRd}\frac{\int_{\D} (\pi
                                    p-p_h)\divergenz_h
                                    \vv_h}{\norm{\nablah\vv_h}_{r}}
    = \sup_{\vv_h\in\CRd}
\frac{\int_{\D}  (p-p_h)\divergenz \EE\vv_h}{\norm{\nablah\vv_h}_{r}}.
  \end{align*}
  For the enumerator, the  continuous equations~\eqref{wp:eq1}
  and Method~\ref{M:DD=nabla} imply
  \begin{align*}
    \int_{\D}  (p-p_h)\divergenz \EE\vv_h
      &=\int_{\D}\left(\A(\nabla
        \uu)-\A(\nablah\uu_h)\right):\nabla\EE\vv_h 
      \\
      &\le \|\A(\nabla \uu)-\A(\nablah\uu_h)\|_{r'}
      \|\nabla\EE\vv_h\|_{r}.
  \end{align*}
  Combining the above estimates with~\eqref{cr:lem1:eq2}  and
  \cite[Lemma 4.6]{Belenki:2012} or \cite[Lemma 2.4]{Hirn:2010} thus
  yields 
  \begin{align*}
    \norm{\pi p-p_h}_{r'} &\lesssim \|\A(\nabla \uu)-\A(\nablah\uu_h)\|_{r'}
    \\
    &\lesssim \begin{cases}
      \norm{\F(\nabla\uu)-\F(\nablah\uu_h)}_2^{\frac{2}{r'}} &\text{if }r\in(1, 2]
      \\
      \norm{\ep+|\nabla\uu|+|\nablah\uu_h|}_r^{\frac{r-2}{2}}
      \norm{\F(\nabla\uu)-\F(\nablah\uu_h)}_2 &\text{if }r\in(2,\infty).
    \end{cases}
  \end{align*}
The first assertion then follows from Theorem~\ref{cr:thm1:eq1}, and,
in the case $r >2$, from the boundedness of Galerkin
solutions~\(\|\nabla \uu\|_r+\|\nablah\uu_h\|_r\lesssim
\|\ff\|_{-1,r'}^{\frac1{r-1}}\), with a hidden constant depending only on
\(r\); compare with~\cite[(2.17)]{Hirn:2013}. The claimed a priori rate follows again from
Theorem~\ref{cr:thm1:eq1} and a Poincar\'{e} inequality.    
\end{proof}

\begin{remark}[Rate of convergence]\label{R:rate-convergence-pressure}
Theorem~\ref{T:pressErr} predicts the same rate of convergence of the
pressure-error as in \cite{Kaltenbach:2023c}, under less restrictive
regularity assumptions and including \(\epsilon=0\); compare also 
with Remark~\ref{R:comparisonV}. The same rate is observed also in the
numerical experiments of \cite{Belenki:2012} where, however, the
theoretical expectation is suboptimal for $r>2$,
cf. \eqref{eq:pressure-error-old}. 
\end{remark}

\section{Numerical experiments}
\label{sec:numerical_experiments}
In this section, we illustrate the performance and the pressure
robustness of the Methods~\ref{M:DD=nabla} and \ref{M:DB=DD1} for the
$r$-Stokes equations, i.e. for  
\begin{equation}\label{p-Stokes}
	\A(\QQ)= |\QQ|^{r-2}\QQ
\end{equation}
in Assumption~\ref{assumption:nonlinear_operator}. Note that the results from \cite{Kaltenbach:2023b,Kaltenbach:2023c} do not apply to this case, as they require the $(r,\ep)$-structure and are not robust in the limit $\varepsilon \to 0.$

To highlight pressure robustness, we compare also with the following counterpart of Method~\ref{M:DD=nabla} without smoothing operator.
\begin{method}\label{M:DD=nabla-without-smoothing}
	For $\ff\in\WW^{-1,r'}(\D)$ compute $\uu_h\in \CRd$ and
	$p_h\in {\wh S}_0^0$ such that
	\begin{alignat*}{2}
		\forall \vv_h&\in\CRd&\qquad\int_{\D} \A(\nablah
		\uu_h):\nablah\vv_h
		- \int_{\D} p_h \divergenz_h\vv_h
		&= \langle \ff, \vv_h\rangle
		,
		\\
		\forall q_h&\in{\wh S}_0^0&\quad\int_{\D} q_h \divergenz_h
		\uu_h
		& =0 .
	\end{alignat*}
\end{method}
\noindent
Recall that we could not apply this method with the gradient replaced by the symmetric gradient, cf. \cite{Arnold:1993}. 

Our implementation is realized with the help of the C library ALBERTA
3.1
\cite{Heine.Koester.Kriessl.Schmidt.Siebert,Schmidt.Siebert:05}. We
solve the nonlinear system of equations stemming from each method via
a relaxed Ka\v{c}anov iteration, based
on~\cite{DieningFornasierTomasiWank:2020,BalciDieningStorn:2023}.
For the
solution of the linearized system, we apply SymmLQ with block diagonal
preconditioning.  

We consider two test cases with prescribed manufactured solution. In
both cases, the domain is a square and the initial mesh $\T_0$ is
obtained by drawing the two main diagonals. We obtain subsequent
meshes $\T_1, \T_2, \dots$ by performing each time two uniform
refinements with newest vertex bisection. We estimate the decay rate
of the velocity- and of the pressure-error with respect to the number
of mesh refinements via the experimental order of convergence,  
\begin{equation*}
	\label{EOC}
	\mathrm{EOC}_k := \log(e_{k-1}/e_k)/\log(2)	
\end{equation*} 
where $e_k$ is the error of interest on $\T_k$ for $k \geq 1$. For
Method~\ref{M:DB=DD1}, we consider e.g. the velocity error with
smoother \(\|\F(\DD\uu)-\F(\DD\EE\uu_h)\|_2\).

\subsection*{Test case 1: power functions}
Following \cite[section~7]{Belenki:2012}, we compute the data so that the solution of \eqref{wp:eq1} with \eqref{p-Stokes} is given by
\begin{equation}
	\label{test-case1}
	\uu(\xx) = |\xx|^{\alpha}(x_2, -x_1) 
	\qquad \text{and} \qquad
	p(\xx) = \eta|\xx|^\gamma  
\end{equation}
for $\xx = (x_1, x_2) \in \Omega = (-1,1)^2$. Different from \eqref{wp:eq1}, we have nonzero Dirichlet datum on $\partial \Omega$, but a similar (undisplayed) test case on the unit circle with zero Dirichlet datum confirmed the results discussed below. In analogy with \cite{Belenki:2012}, we set 
\begin{align*}
\alpha = 0.01, 
\qquad
\gamma = \frac{2}{r} - 1 + 0.01, \qquad 
\eta = \begin{cases}
	0.01 & \text{if } r \in (1,2),\\
	1    & \text{if } r \in [2,\infty).
\end{cases}
\end{align*}
The factor \(\eta\) scales the pressure best-error relative to the
velocity-error  on the right hand
side of~\eqref{eq:presRate} for
\(r\in(1,2)\) such that the suboptimal convergence of the
pressure-error can be better observed in Figure~\ref{F:eoc_pre_power}.  

\begin{figure}[htp]
	\captionsetup[subfigure]{labelformat=empty}
	\hfill
	\subfloat[Method~\ref{M:DD=nabla}]{\includegraphics[width=0.32\hsize]{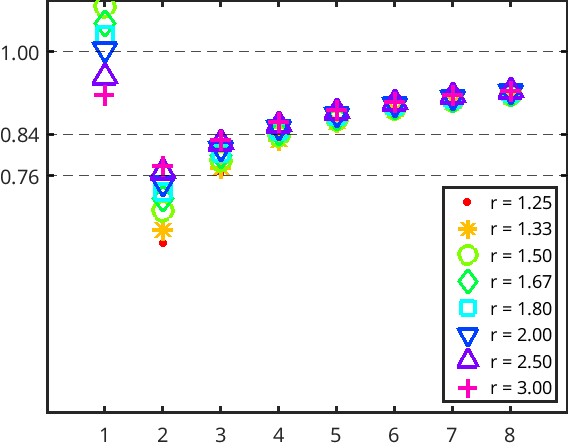}}
	\hfill
	\subfloat[Method~\ref{M:DB=DD1}]{\includegraphics[width=0.32\hsize]{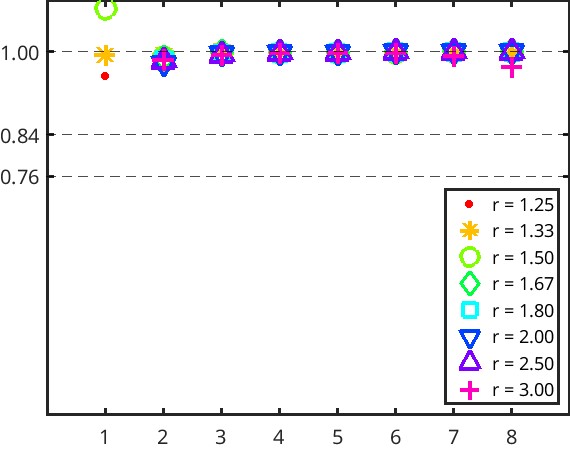}}
	\hfill
	\subfloat[Method~\ref{M:DD=nabla-without-smoothing}]{\includegraphics[width=0.32\hsize]{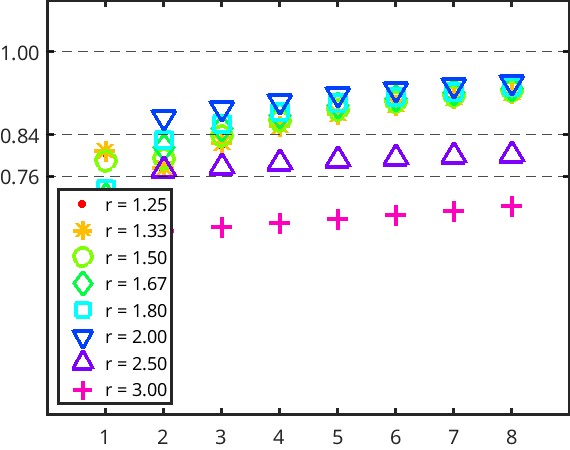}}
	\hfill
	\caption{Test case 1. Experimental order of convergence $\mathrm{EOC}_k$ of the velocity-error versus the number $k$ of mesh refinements.} 
	\label{F:eoc_vel_power}
\end{figure}

\begin{figure}[htp]
	\captionsetup[subfigure]{labelformat=empty}
	\hfill
	\subfloat[Method~\ref{M:DD=nabla}]{\includegraphics[width=0.32\hsize]{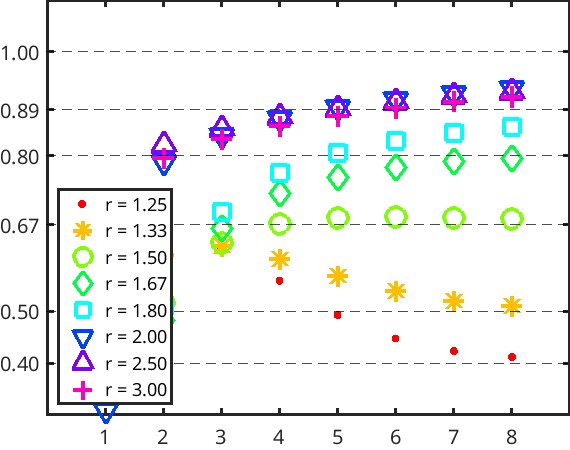}}
	\hfill
	\subfloat[Method~\ref{M:DB=DD1}]{\includegraphics[width=0.32\hsize]{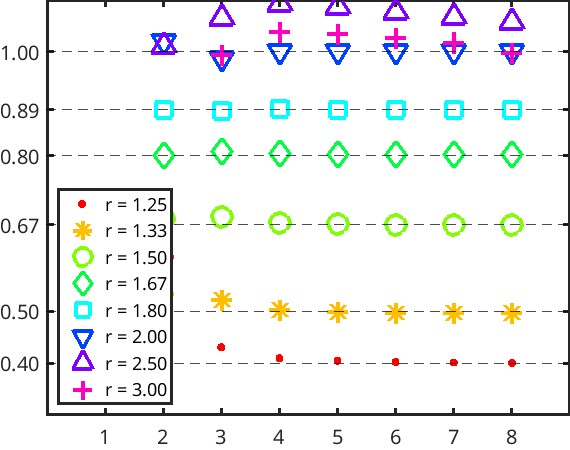}}
	\hfill
	\subfloat[Method~\ref{M:DD=nabla-without-smoothing}]{\includegraphics[width=0.32\hsize]{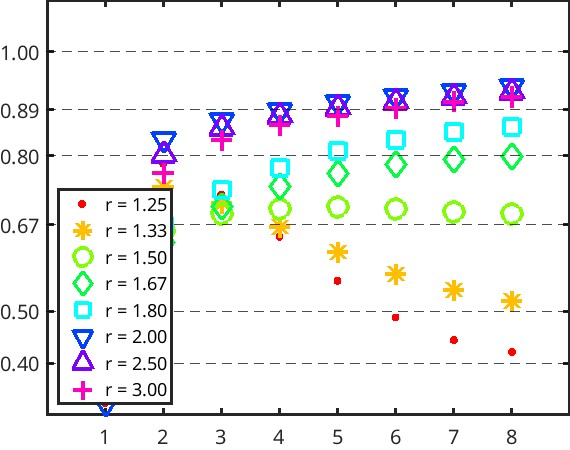}}
	\hfill
	\caption{Test case 1. Experimental order of convergence $\mathrm{EOC}_k$ of the pressure-error versus the number $k$ of mesh refinements.} 
	\label{F:eoc_pre_power}
\end{figure}

Since $\F(\DB\uu) \in
W^{1,2}(\Omega)^{2\times 2}$, we expect that $\mathrm{EOC}_k$ for the
velocity-error of both Method~\ref{M:DD=nabla} and \ref{M:DB=DD1}
converges to $1$ as $k$ increases, owing to Theorems~\ref{cr:thm1}
and~\ref{T:aprioriCRDD1}, respectively. Similarly, since $p \in
W^{1,r'}(\Omega)$ for all $r > 1$, we expect that $\mathrm{EOC}_k$ for
the pressure-error of both methods converges to $\min\{1,
\frac{2}{r'}\}$, as a consequence of Theorem~\ref{T:pressErr}. The
data in Figure~\ref{F:eoc_vel_power} and~\ref{F:eoc_pre_power} (left
and center panel) confirm our expectation. The convergence to the expected value is faster for Method~\ref{M:DB=DD1} than for Method~\ref{M:DD=nabla}. We verified that the gap is somehow related to this specific example and it is less evident if, e.g., we increase the exponent $\alpha$ in \eqref{test-case1}, see also the next test case.

Regarding Method~\ref{M:DD=nabla-without-smoothing}, we observe that
$\mathrm{EOC}_k$ for the velocity-error converges to the suboptimal
value $\min\{1, \frac{r'}{2}\}$ (Figure~\ref{F:eoc_vel_power}, right
panel), which is in line with the estimate \eqref{intro:eq4} for
another method lacking pressure robustness. For the pressure-error
(Figure~\ref{F:eoc_pre_power}, right panel), the results suggest convergence of $\mathrm{EOC}_k$ to $\min\{1, \frac{2}{r'}\}$, as for the other methods. Thus, for $r>2$, the convergence is faster than expected from \eqref{eq:pressure-error-old}, as reported also in \cite{Belenki:2012}.

\subsection*{Test case 2: jumping pressure}
In analogy with \cite[section~6.3]{VerfuerthZanotti:2019} and \cite[section~4.2]{KreuzerVerfuerthZanotti:2021}, we consider a test case with smooth velocity and discontinuous pressure
\begin{equation*}
	\label{test-case2}
	\uu(\xx) = |\xx|^{0.5}(x_2, -x_1) 
	\qquad \text{and} \qquad
	p(\xx) = \begin{cases}
		-\frac{3}{2}, & x_1 < \frac{2}{3}\\
		3, & x_1 > \frac{2}{3}
	\end{cases}
\end{equation*}
for $\xx = (x_1, x_2) \in \Omega = (0,1)^2$. As before, we expect that $\mathrm{EOC}_k$ for the velocity-error of both Method~\ref{M:DD=nabla} and~\ref{M:DB=DD1} converges to $1$ as $k$ increases. In contrast, since $p \in W^{s,r'}(\Omega)$ only for $s < \frac{1}{r'}$, we expect that $\mathrm{EOC}_k$ for the pressure-error of both methods converges to $\frac{1}{r'}$ for all $r > 1$. Our expectation is confirmed by the data in Figure~\ref{F:eoc_vel_pdisc} and~\ref{F:eoc_pre_pdisc} (left and center panel). 

\begin{figure}[htp]
	\captionsetup[subfigure]{labelformat=empty}
	\hfill
	\subfloat[Method~\ref{M:DD=nabla}]{\includegraphics[width=0.32\hsize]{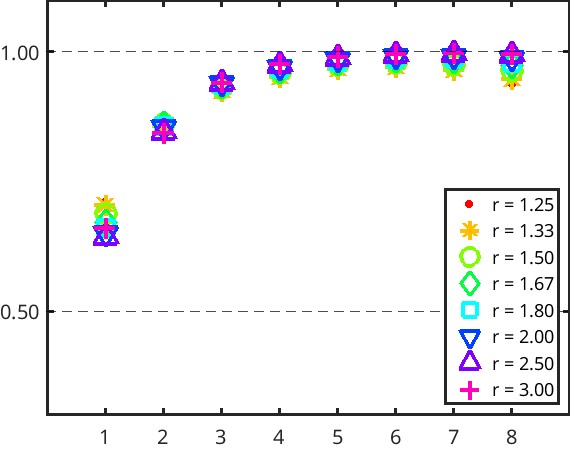}}
	\hfill
	\subfloat[Method~\ref{M:DB=DD1}]{\includegraphics[width=0.32\hsize]{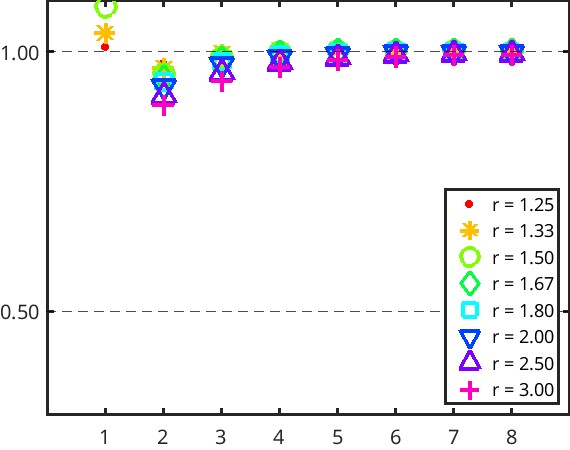}}
	\hfill
	\subfloat[Method~\ref{M:DD=nabla-without-smoothing}]{\includegraphics[width=0.32\hsize]{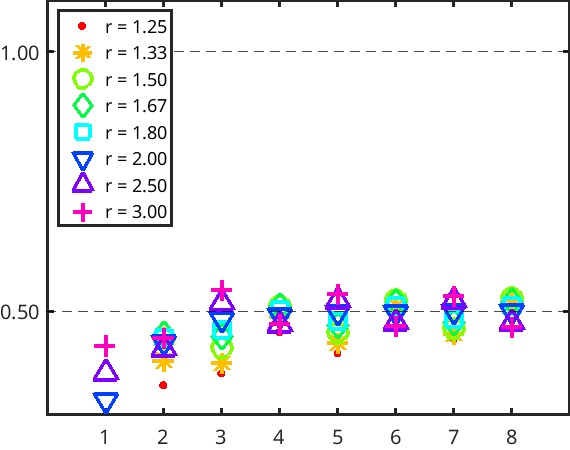}}
	\hfill
	\caption{Test case 2. Experimental order of convergence $\mathrm{EOC}_k$ of the velocity-error versus the number $k$ of mesh refinements.} 
	\label{F:eoc_vel_pdisc}
\end{figure}

\begin{figure}[htp]
	\captionsetup[subfigure]{labelformat=empty}
	\hfill
	\subfloat[Method~\ref{M:DD=nabla}]{\includegraphics[width=0.32\hsize]{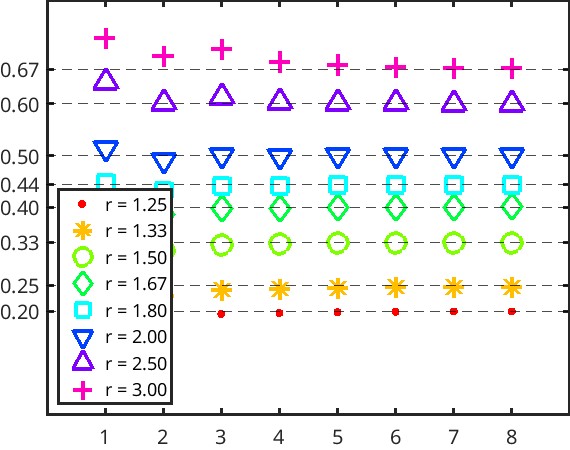}}
	\hfill
	\subfloat[Method~\ref{M:DB=DD1}]{\includegraphics[width=0.32\hsize]{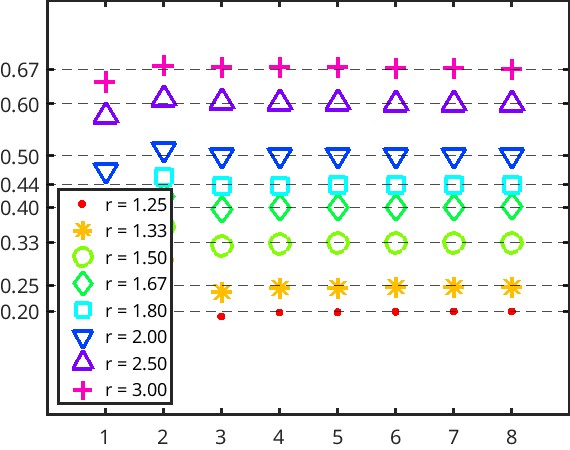}}
	\hfill
	\subfloat[Method~\ref{M:DD=nabla-without-smoothing}]{\includegraphics[width=0.32\hsize]{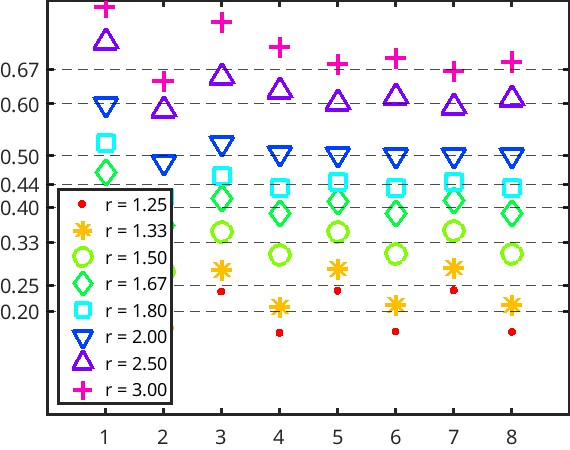}}
	\hfill
	\caption{Test case 2. Experimental order of convergence $\mathrm{EOC}_k$ of the pressure-error versus the number $k$ of mesh refinements.} 
	\label{F:eoc_pre_pdisc}
\end{figure}

Compared with the previous test case, the pressure regularity is much lower and the importance of pressure robustness is even more apparent. Indeed, for the Method~\ref{M:DD=nabla-without-smoothing} (Figure~\ref{F:eoc_vel_pdisc}, right panel), we observe that $\mathrm{EOC}_k$ for the velocity-error converges to the sub-optimal rate $0.5$, irrespective of $r$, whereas for the pressure-error we have similar results as for the other methods (Figure~\ref{F:eoc_pre_pdisc}, right panel).

\section{Conclusions}
\label{sec:conclusions}

We have proposed two numerical methods for the nonlinear Stokes equations~\eqref{intro:eq1} using nonconforming
Crouzeix-Raviart velocity elements and piecewise constant pressure elements (Method~\ref{M:DD=nabla} and
\ref{M:DB=DD1}) 
with \(\DB=\nabla\) and \(\DB=\DD\) respectively. Both schemes feature a
monotone nonlinear operator acting on the velocity and thus the
discrete problems are uniquely solvable; see
Propositions~\ref{P:M-DB=nabla} and ~\ref{P:M-DB=DD1}.
Both methods are
pressure robust. The error estimates in
Theorems~\ref{cr:thm1:eq1},~\ref{T:aprioriCRDD1} and~\ref{T:pressErr}
improve upon existing results for similar methods and are in line with
the best known results for conforming and divergence-free methods.

\subsection*{Funding}
Christian Kreuzer acknowledges funding by the Deutsche
For\-schungs\-ge\-mein\-schaft (DFG, German Research Foundation) --
321270008. Lars Diening research is funded by the Deutsche Forschungsgemeinschaft (DFG, German Research Foundation) - SFB 1283/2 2021.
Pietro Zanotti was supported by the GNCS-INdAM project CUP E53C23001670001.

\printbibliography

\end{document}